\documentclass[10pt]{amsart}
\usepackage{graphicx}
\usepackage{amscd}
\usepackage[dvipsnames]{xcolor}
\usepackage{amsmath}
\usepackage{amsthm}
\usepackage{amsfonts}
\usepackage{amssymb}
\usepackage{mathrsfs}
\usepackage{enumerate}
\usepackage[pagebackref,hypertexnames=false, colorlinks, citecolor=blue, linkcolor=red, pdfstartview=FitB]{hyperref}
\usepackage[latin1]{inputenc}
\newcommand{\bB}{{\mathbb{B}}}
\newcommand{\bC}{{\mathbb{C}}}
\newcommand{\bD}{{\mathbb{D}}}

\newcommand{\bN}{{\mathbb{N}}}

\newcommand{\bR}{{\mathbb{R}}}
\newcommand{\bS}{{\mathbb{S}}}
\newcommand{\bT}{{\mathbb{T}}}

  \newcommand{\A}{{\mathcal{A}}}
  \newcommand{\B}{{\mathcal{B}}}
  \newcommand{\C}{{\mathcal{C}}}
  
  \newcommand{\E}{{\mathcal{E}}}

\renewcommand{\H}{{\mathcal{H}}}

  \newcommand{\M}{{\mathcal{M}}}

  \newcommand{\R}{{\mathcal{R}}}
\renewcommand{\S}{{\mathcal{S}}}

  \newcommand{\W}{{\mathcal{W}}}
  \newcommand{\X}{{\mathcal{X}}}
  \newcommand{\Y}{{\mathcal{Y}}}

    \newcommand{\cA}{{\mathcal{A}}}

  \newcommand{\cR}{{\mathcal{R}}}

\newcommand{\fJ}{{\mathfrak{J}}}
\newcommand{\fK}{{\mathfrak{K}}}

\newcommand{\fT}{{\mathfrak{T}}}

\newcommand{\fz}{{\mathfrak{z}}}

\newcommand{\rA}{\mathrm{A}}

\newcommand{\rC}{\mathrm{C}}

\renewcommand{\phi}{\varphi}

\newcommand{\upchi}{{\raise.35ex\hbox{$\chi$}}}

\newcommand{\ol}{\overline}

\newcommand{\AB}{{\mathrm{A}(\mathbb{B}_d)}}
\newcommand{\AD}{{\mathrm{A}(\mathbb{D})}}

\newcommand{\Hinf}{{H^\infty(\mathbb{D})}}

\newcommand{\qand}{\quad\text{and}\quad}

\newcommand{\id}{\operatorname{id}}

\newcommand{\re}{\operatorname{Re}}

\newcommand{\AC}{\operatorname{AC}}
\newcommand{\ac}{\operatorname{AC}}
\newcommand{\SG}{\operatorname{SG}}
\newcommand{\sing}{\operatorname{SG}}

\newcommand{\Han}{\operatorname{Han}}

\let\phi\varphi
\newtheorem{lemma}{Lemma}[section]
\newtheorem{theorem}[lemma]{Theorem}
\newtheorem{proposition}[lemma]{Proposition}
\newtheorem{corollary}[lemma]{Corollary}
\newtheorem{theoremx}{Theorem}

\theoremstyle{definition}

\newtheorem{example}[lemma]{Example}

\newtheorem{remark}[lemma]{Remark}

\title[Lebesgue decompositions and the Gleason--Whitney property]{Lebesgue decompositions and the Gleason--Whitney property for operator algebras}

\author{Rapha\"el Clou\^atre}

\address{Department of Mathematics, University of Manitoba, Winnipeg, Manitoba, Canada R3T 2N2}

\email{raphael.clouatre@umanitoba.ca\vspace{-2ex}}
\author{Michael Hartz}
\address{Fachrichtung Mathematik, Universit\"at des Saarlandes, 66123 Saarbr\"ucken,Germany}

\email{hartz@math.uni-sb.de\vspace{-2ex}}
\thanks{R.C. was partially supported by an NSERC Discovery Grant. M.H. was partially supported by a GIF grant and by the Emmy Noether Program of the German Research Foundation (DFG Grant 466012782).}
\date{\today}
\begin{document}
\begin{abstract}
Broadly speaking, this paper is concerned with dual spaces of operator algebras. More precisely, we investigate the existence of what we call Lebesgue projections: central projections in the bidual of an operator algebra that detect the weak-$*$ continuous part of the dual space. Associated to any such projection is a Lebesgue decomposition of the dual space. We are particularly interested in Lebesgue projections in the context of inclusions of operator algebras. We show how their presence is intimately connected with an extension property for the inclusion reminiscent of a classical theorem of Gleason and Whitney.
We illustrate that this Gleason--Whitney property fails in many examples of concrete operator algebras of functions, which partly explains why compatible Lebesgue decompositions are scarce, and highlights that the classical inclusion $H^\infty\subset L^\infty$ on the circle does not display generic behaviour.  \end{abstract}

\maketitle

\section{Introduction}
Let $X$ be a compact Hausdorff space and let $\lambda$ be a regular Borel probability measure on $X$. The classical Lebesgue decomposition expresses an arbitrary finite Borel measure $\mu$ on $X$ as  $\mu_a+\mu_s$, where $\mu_a$ is absolutely continuous with respect to $\lambda$ and $\mu_s$ is singular with respect to $\lambda$. Equivalently, there is a closed subspace $S$ with the property that
\[
\rC(X)^*=A\oplus_1 S
\]
where $A\subset \rC(X)^*$ consists of those functionals admitting a weak-$*$ continuous extension to $L^\infty(X,\lambda)$. 

Such a \emph{Lebesgue decomposition} exists much more generally. Let $\W$ be a von Neumann algebra and let $\fT\subset \W$ be a weak-$*$ dense $\rC^*$-subalgebra. Let $\ac_\W(\fT)\subset \fT^*$ denote the subspace consisting of functionals admitting a weak-$*$ continuous extension to $\W$; it can be verified that this space is closed. The space $\ac_\W(\fT)^\perp$ is readily seen to be a weak-$*$ closed two-sided ideal of the von Neumann algebra $\fT^{**}$. As such, there is a unique central projection $\fz\in \fT^{**}$ with the property that  \[\fT^*=\ac_\W(\fT)\oplus_1 (\fz \fT^{**})_\perp\] (see Section \ref{S:OA} for details). We refer to $\fz$ as the \emph{Lebesgue projection} for $\fT$ relative to $\W$. 
In the language of Banach space theory, see \cite{HWW}, this says that $\ac_\W(\fT)$ is an $L$-summand (also known as $\ell^1$-direct summand) of $\fT^*$. 

We caution the reader that what we call Lebesgue decomposition is different from the objects considered in \cite{henle1972},\cite{exel1990} or \cite{GK2009}: whereas those papers are concerned with absolute continuity with respect to a given state or completely positive map, here we are interested in weak-$*$ continuous extensions. 

A main theme in this paper is the existence of Lebesgue decompositions and Lebesgue projections for general, possibly non-selfadjoint, operator algebras. In this non-selfadjoint word, the argument above is not available, and as such Lebesgue decompositions are much harder to come by. Such difficulties have prompted recent studies of absolutely continuous functionals through their manifestations as Henkin functionals \cite{CD2016duality},\cite{BHM2018},\cite{DH2020},\cite{CT2022ncHenkin}.

Concretely, we consider the following setup. Let $\mathcal{W}$ be a von Neumann algebra and let $\mathcal{A} \subset \mathcal{W}$ be a unital (not necessarily self-adjoint) subalgebra.
Let $\AC_{\mathcal{W}}(\mathcal{A}) \subset \mathcal{A}^{*}$ be the subspace
of all functionals that extend to weak-$*$ continuous functionals on $\mathcal{W}$.
We say that $\mathcal{A}$ admits a \emph{Lebesgue decomposition} if $\AC_{\mathcal{W}}(\mathcal{A})$ is a norm-closed $L$-summand of $\mathcal{A}^*$; see Section \ref{S:compat} for more details.

One appeal of Lebesgue decompositions is that they can be used to establish uniqueness of preduals for dual spaces \cite{pfitzner2007}. Another use for them is the study of Hilbert space operators through well-behaved functional calculi \cite{nagy2010},\cite{CD2016abscont},\cite{BHM2018}. Underlying the latter idea is the fact that some concrete algebras of functions are known to admit a Lebesgue decomposition. Indeed, this is true of  the ball algebra \cite[Chapter 9]{rudin2008}, of the norm closure of the polynomials in the multiplier algebra of the Drury--Arveson space \cite{CD2016duality} and of many other norm closed algebras of multipliers of function spaces on the ball \cite{DH2020}. In that vein, it should also be noted that dual spaces of certain non-commutative operator algebras of functions have been shown to possess a certain decomposition property that is weaker than what we call a Lebesgue decomposition \cite[Proposition 5.9]{davidson2005},\cite{JM2019}.

Other instances of spaces admitting a Lebesgue decompositon include certain essential commutants \cite{voiculescu2017}, as well as the classical algebra $\Hinf$ \cite{ando1978} and some of its natural non-commutative analogues \cite{ueda2009}. It is relevant to mention that in all those cases, the results were established by leveraging the Lebesgue decompositions of the corresponding ambient von Neumann algebras. To determine whether this strategy can be employed for general inclusions of operator algebras is one of the guiding principles of the paper, which is organized as follows.

In Section \ref{S:compat}, we introduce our main objects of study: Lebesgue decompositions. Of particular importance for our purposes is the notion of compatibility for pairs of decompositions. Theorem \ref{T:Lebcompat} identifies necessary and sufficient conditions for compatibility to hold. 

In Section \ref{S:OA}, we examine Lebesgue decompositions in the setting of operator algebras. We show in Theorem \ref{T:Lebproj} that a Lebesgue decomposition behaves well with respect to the multiplicative structure and arises via a Lebesgue projection in the second dual; this is related to the notion of \emph{inner $M$-ideals} (see \cite[Section V.3]{HWW}).

Section \ref{S:GW} is centred around a property introduced by Blecher and Labuschagne in \cite{BL2007},\cite{blecherlabuschagne2007},\cite{BL2018} which we briefly recall here. Let $\W$ be a von Neumann algebra. Let $\B\subset \W$ be unital operator algebra and let $\A\subset \B$ be a unital subalgebra. We say that the inclusion $\A\subset\B$ has the \emph{Gleason--Whitney property} relative to $\W$ if, whenever a bounded linear functional on $\A$ admits a weak-$*$ continuous extension to $\W$, then so do all of its Hahn--Banach extensions to $\B$. We explore the relationship between the Gleason--Whitney property and the existence of Lebesgue decompositions and obtain the following, which is one of our main results and is found in Theorem \ref{T:GWcompat}.

\begin{theoremx}\label{T:main1}
  Assume that $\A$ and $\B$  both admit Lebesgue decompositions and let $\fz_\A\in \A^{**}$ and $\fz_\B\in \B^{**}$ denote the corresponding Lebesgue projections.   Then, the following statements are equivalent.
\begin{enumerate}[{\rm (i)}]

\item The inclusion $\A\subset \B$ has the  Gleason--Whitney property relative to $\W$.
\item We have $\fz_\A=\fz_\B$.
\item The Lebesgue decompositions of $\A$ and $\B$ are compatible.
\end{enumerate}
\end{theoremx}

In Section \ref{S:obstruction}, we study the possibility of
obtaining a version of Theorem \ref{T:main1} without assuming a priori that $\mathcal{A}$
admits a Lebesgue decomposition. In particular,
a significant strengthening of the implication (i) $\Rightarrow$ (ii)
of Theorem \ref{T:main1} would be the statement that
if $\mathcal{B}$ admits a Lebesgue decomposition and $\mathcal{A} \subset \mathcal{B}$ has the Gleason--Whitney
property, then the Lebesgue projection of $\mathcal{B}$ belongs to $\mathcal{A}^{**}$.
In turn, this would imply the existence of a Lebesgue decomposition of $\mathcal{A}$ that is
compatible with that of $\mathcal{B}$; see Proposition \ref{prop:Lebesgue_proj_smaller_alg}.
For instance, this would make it possible to deduce And\^o's Lebesgue decomposition
of the dual of $H^\infty(\mathbb{D})$ from the Lebesgue decomposition of the dual
of $L^\infty$ and the classical Gleason--Whitney theorem.
We were unable to prove such a strengthening.
As a step towards this goal, we establish Theorem \ref{T:Lebesgue_projection_character}, which contains the following result as a special case.

\begin{theoremx}\label{T:main2}
  Let $\mathcal{H}$ be a Hilbert space and let $\fT\subset B(\H)$ be a unital $\rC^*$-algebra containing the compact operators. Let $\A\subset \fT$ be a unital subalgebra such that the inclusion $\A\subset \fT$ has the Gleason--Whitney property relative to $B(\H)$. Then $\fT$ admits a non-trivial Lebesgue projection $\fz$, and for any absolutely continuous character $\chi$ on $\mathcal{A}$, there is a non-trivial projection $p \in \A^{**}$ such that $p \le \fz$ and $p(\chi) = 1$.
\end{theoremx}
This result can also be interpreted as an obstruction to the Gleason--Whitney property
of an inclusion $\mathcal{A} \subset \mathcal{B}$, since $\mathcal{A}^{**}$ must contain
a non-trivial projection.

Section \ref{S:examples} contains many detailed and non-trivial examples illustrating the previous ideas.
The following result summarizes some of these examples.

\begin{theoremx}
  The following inclusions of algebras do not satisfy the Gleason--Whitney property (relative to the second algebra, unless otherwise stated):
  \begin{enumerate}[\rm (a)]
    \item $C([0,1]) \subset L^\infty([0,1])$;
    \item  $L^\infty([0,1]) \subset B(L^2([0,1]))$;
    \item $\rA(\mathbb{D}) \subset H^\infty(\mathbb{D})$ relative to $L^\infty(\mathbb{T})$;
    \item $H^\infty(\mathbb{D}) \subset B(H^2(\mathbb{D}))$;
    \item $H^\infty(\mathbb{D}) \subset L^\infty(\mathbb{D})$;
    \item $H^\infty(\mathbb{B}_d) \subset L^\infty(\partial \mathbb{B}_d)$ if $d \ge 2$;
    \item $H^\infty(\mathbb{D}^d) \subset L^\infty(\mathbb{T}^d)$ if $d \ge 2$.
  \end{enumerate}
\end{theoremx}

\begin{proof}
  (a) is a special case of Proposition \ref{P:C(X)GW}.

  (b) is Example \ref{E:Linfty}.

  (c) is Proposition \ref{P:GWADHinf}.

  (d) follows from Proposition \ref{P:MFTF}.

  (e) is Proposition \ref{HinfbD}.

  (f) is Proposition \ref{P:HballGW}.

  (g) is Proposition \ref{P:HpolydiscGW}.
\end{proof}

Taken together, these examples make a compelling case that the Gleason--Whitney property is rather rare, at least among commonly studied inclusions of operator algebras coming from function theory and  multivariate operator theory.  
By contrast, an earlier positive result (Corollary \ref{C:compact}) applies in particular to the space of Hankel operators on a reproducing kernel Hilbert space.

Finally, in Section \ref{S:homom} we study absolute continuity for the more structured class of homomorphisms on operator algebras. Theorem \ref{T:staterep} shows how this issue is closely related to the corresponding one for states. We also exhibit a connection between our results and the classical problem of existence of pathological functional calculi, as pioneered in the deep work of Miller--Olin--Thompson \cite{MOT86}.

\section{Lebesgue decompositions and compatibiity}\label{S:compat}

Throughout the paper, given a normed space $\X$ we will tacitly identify it with its canonical image in $\X^{**}$. Likewise, if $\Y\subset \X$ is a subspace, then we regard $\Y^{**}$ as a subspace of $\X^{**}$ upon identifying with $\Y^{\perp\perp}$.

\subsection{Lebesgue decompositions of normed spaces}\label{SS:LebdecompBspace}
In some places, we will work in a setting that is more general than the one described in the introduction.
We now describe this setting.
Let $\W$ be the dual space of some normed space. Let  $\X\subset \W$  be a subspace. A bounded linear functional $\phi:\X\to\bC$ is said to be \emph{absolutely continuous relative to} $\W$ if there is a weak-$*$ continuous linear map $\Phi:\W\to \bC$ extending $\phi$. We denote by $\ac_{\W}(\X)$ the subspace of $\X^*$ comprising all absolutely continuous functionals relative to $\W$. Often, the space $\W$ will be clear from context and we simply speak of absolutely continuous functionals.

We say that $\X$ admits a \emph{Lebesgue decomposition relative to $\W$} if there is a bounded linear idempotent $L:\X^*\to \X^*$ with the property that $\ac_\W(\X)=L\X^*$ and such that
\[
\|\phi\|=\|L\phi\|+\|(I-L)\phi\|, \quad \phi\in \X^*.
\]
In other words, $\X$ admits a Lebesgue decomposition whenever $\ac_\W(\X)$ is an \emph{$L$-summand} of $\X^*$, see \cite{HWW} for background on this notion.
Equivalently, this is saying that $\ac_\W(\X)$ is a norm-closed subspace of $\X^*$ and that there is a subspace $\S\subset \X^*$ with the property that
\begin{equation}\label{Eq:LD}
\X^*=\ac_\W(\X)\oplus_1 \S.
\end{equation}
We will also refer to \eqref{Eq:LD} as the Lebesgue decomposition.
The following sufficient condition for norm-closedness of $\AC_{\mathcal{W}}(\mathcal{X})$
is standard. In some concrete cases, it can also be deduced from \cite[Lemma 1.1]{eschmeier1997}, \cite[Lemma 3.1]{BHM2018}, or \cite[Proposition 4.1]{CT2022ncHenkin}.

\begin{proposition}\label{P:ACclosed}
  Let $\mathcal{W}$ be the dual space of some normed space and let $\mathcal{X} \subset \mathcal{W}$.
If the unit ball of $\mathcal{X}$ is weak-$*$ dense in the unit ball of the weak-$*$ closure of $\mathcal{X}$ in $\mathcal{W}$, then $\AC_{\mathcal{W}}(\mathcal{X})$ is norm-closed in $\mathcal{X}^*$.
\end{proposition}

\begin{proof}
  By replacing $\mathcal{W}$ with the weak-$*$ closure of $\mathcal{X}$ in $\mathcal{W}$,
  we may assume that the unit ball of $\mathcal{X}$ is weak-$*$ dense in the unit ball of $\mathcal{W}$.
  Let $\mathcal{W}_* \subset \mathcal{W}^*$ be the space of weak-$*$ continuous functionals
  on $\mathcal{W}$.
  Then
  \begin{equation*}
    \mathcal{W}_* \to \mathcal{X}^*, \quad \varphi \mapsto \varphi \big|_{\mathcal{X}},
  \end{equation*}
  is an isometry whose range equals $\AC_{\mathcal{W}}(\mathcal{X})$.
  Since $\mathcal{W}_*$ is complete, the space $\AC_{\mathcal{W}}(\mathcal{X})$
  is complete as well and hence closed.
\end{proof}

We remark that the subspace $\S$ appearing in \eqref{Eq:LD} is uniquely determined \cite[Proposition I.1.2]{HWW}.  Functionals in $\S$ will be said to be \emph{singular relative to $\W$} and we will write $\S=\sing_\W(\X)$. Given $\phi\in \X^*$, we use the notations
\[
\phi_a=L\phi \qand \phi_s= (I-L)\phi.
\]
Singular functionals can be characterized intrinsically as follows.
A special case of the following result already appeared as \cite[Lemma 4.1]{DH2020}.

\begin{proposition}\label{P:singnorm}
 Let $\W$ be the dual space of some normed space and let $\X\subset \W$ be a subspace which admits a Lebesgue decomposition relative to $\W$. Let $\phi\in \X^*$. Then, $\phi\in \sing_\W(\X)$ if and only if there exists
  a net $(x_\alpha)$ in the unit ball of $\X$ that converges to $0$ in the weak-$*$ topology of $\W$ and such that $(\varphi(x_\alpha))$ converges
  to $\|\varphi\|$. 
  If the predual of $\mathcal{W}$ is separable, then the net $(x_\alpha)$ can be replaced with a sequence.
\end{proposition}

\begin{proof}
Suppose first that there exists a net $(x_\alpha)$ in the unit ball of $\X$
  which converges to $0$ in the weak-$*$ topology of $\W$ and satisfies $\lim_\alpha \varphi(x_\alpha) = \|\varphi\|$.
Write $\varphi = \varphi_a + \varphi_s$ with $\phi_a\in \ac(\X), \phi_s\in \sing(\X)$. We find
  \begin{equation*}
    \lim_\alpha (\varphi(x_\alpha) - \varphi_s(x_\alpha))= \lim_\alpha \varphi_a(x_\alpha) =0
  \end{equation*}
  whence
  \[
  \lim_\alpha\phi_s(x_\alpha)=\|\phi\|.
  \]
 We infer that $\|\phi_s\|\geq \|\phi\|$. Since $\|\varphi\| = \|\varphi_a\| + \|\varphi_s\|$, it follows that $\varphi_a = 0$ and $\varphi = \varphi_s\in \sing(\X)$.

  Conversely, let $\varphi \in \sing_\W(\X)$. By the Hahn--Banach theorem, there exists
  a  $\Lambda\in \X^{**}$ with norm $1$ such that $\Lambda(\varphi) = \|\varphi\|$.  Upon replacing $\Lambda$ with $\Lambda\circ (I-L)$ if necessary, we may assume that $\ac_\W(\X)\subset \ker \Lambda$. Now, by Goldstine's theorem, there
  exists a net $(x_\alpha)$ in the unit ball of $\X$ such that $(x_\alpha)$  converges to $\Lambda$
  in the weak-$*$ topology of $\X^{**}$. Then 
  \[
  \lim_\alpha \varphi(x_\alpha)=\Lambda(\varphi) = \|\varphi\|,
  \]
  so it remains to show that $(x_\alpha)_\alpha$ converges to $0$ in the weak-$*$ topology of $\W$.
   To this end, let  $\omega$ be a weak-$*$ continuous functional on $\W$.  Then, $\omega|_\X\in \ac_\W(\X)$ so that 
   \[
   0=\Lambda(\omega|_\X)=\lim_\alpha \omega( x_\alpha),
   \]
and the proof of the equivalence is complete.

Finally, suppose that $\mathcal{W}$ is the dual space of a separable normed space.
Then the weak-$*$ topology on the unit ball of $\mathcal{W}$ is given by a metric $d$.
By the preceding paragraph, there exists for each $n \ge 1$ an element $x_n$ in the unit ball of $\mathcal{X}$
such that $d(x_n,0) < \frac{1}{n}$ and $|\varphi(x_n) - \|\varphi\| | < \frac{1}{n}$.
Then $(x_n)$ converges to zero in the weak-$*$ topology of $\mathcal{W}$ and $\lim_{n \to \infty} \varphi(x_n) = \|\varphi\|$.
\end{proof}

\subsection{Compatibility of Lebesgue decompositions}\label{SS:compat}
Let $\W$ be the dual space of some normed space, and let $\X\subset \Y$ be a pair of subspaces of $\W$. Assume that both $\X$ and $\Y$ admit a Lebesgue decomposition. We will say that the Lebesgue decompositions of $\X$ and $\Y$ are \emph{compatible} if, given a functional $\psi\in \Y^*$ that we decompose as $\psi=\psi_a+\psi_s$ with $\psi_a\in \ac_\W(\Y)$ and $\psi_s\in \sing_\W(\Y)$, then we must have that $\psi_a|_\X\in \ac_\W(\X)$ and $\psi_s|_\X\in \sing_\W(\X)$. In other words, if we let $\phi=\psi|_\X$ and write $\phi=\phi_a+\phi_s$ with $\phi_a\in \ac_\W(\X)$ and $\phi_s\in \sing_\W(\X)$, then we have that $\phi_a=\psi_a|_\X$ and $\phi_s=\psi_s|_\X$. Clearly, the condition $\psi_a|_\X\in \ac_\W(\X)$ is automatically satisfied, so compatibility of the Lebesgue decompositions boils down to whether singular functionals on $\Y$ restrict to singular functionals on $\X$; this observation will be utilized below. 

We show next that the converse statement is always true, at least for norm-preserving extensions. Let us first set up some terminology used throughout the paper. Given $\phi\in \X^*$, a functional $\psi\in \Y^*$ is a \emph{Hahn--Banach extension} of $\phi$ if $\psi|_\X=\phi$ and $\|\phi\|=\|\psi\|$. In other words, a Hahn--Banach extension is simply a norm-preserving extension, which always exists by the Hahn--Banach theorem.

\begin{lemma}\label{L:HBsing}
Let $\phi\in \sing_\W(\X)$ and let $\psi\in \Y^*$ be a Hahn--Banach extension of $\phi$. Then, $\psi\in \sing_\W(\Y)$.
\end{lemma}
\begin{proof}
It follows from Proposition \ref{P:singnorm} that there exists a net $(x_\alpha)$ in the unit ball of $\X$ that converges to $0$ in the weak-$*$ topology of $\W$ and such that $(\varphi(x_\alpha))$ converges
  to $\|\varphi\|$. Since $\|\phi\|=\|\psi\|$, the same lemma then implies in turn that $\psi\in \SG_\W(\Y)$.
\end{proof}

We can now reformulate the compatibility condition of Lebesgue decompositions. 

\begin{theorem}\label{T:Lebcompat}
Let $\W$ be the dual space of some normed space, and let $\X\subset \Y$ be a pair of subspaces of $\W$.
Assume that $\Y$ admits a Lebesgue decomposition. Then, the following statements are equivalent.
  \begin{enumerate}[{\rm (i)}]
         \item The space $\X$ admits a Lebesgue decomposition that is compatible with that of $\Y$.

        \item Let $\psi\in  \X^\perp.$ If we write $\psi= \psi_a + \psi_s$ with $\psi_a\in \ac_\W(\Y)$ and $\psi_s\in \sing_\W(\Y)$, then $\psi_a\in \X^\perp$ and $\psi_s\in \X^\perp$.
     \end{enumerate}
\end{theorem}

\begin{proof}
   (i) $\Rightarrow$ (ii): Let $\psi\in \Y^*$, which we write as $\psi=\psi_a+\psi_s$ with $\psi_a\in \ac_\W(\Y)$ and $\psi_s\in \sing_\W(\Y)$. Further,  put $\phi=\psi|_\X$ and write $\phi=\phi_a+\phi_s$ with $\phi_a\in \ac_\W(\X)$ and $\phi_s\in \sing_\W(\X)$. Since the Lebesgue decompositions are assumed to be compatible,  we see that $\phi_a=\psi_a|_\X$ and $\phi_s=\psi_s|_\X$. If we assume in addition that $\psi\in \X^\perp$, then $\phi=0$, so $\psi_a|_\X=\phi_a=0$ and $\psi_s|_\X=\phi_s=0$. 

  (ii) $\Rightarrow$ (i):
  Let $L: \mathcal{Y}^* \to \mathcal{Y}^*$ be the idempotent in the definition of Lebesgue decomposition.
  As usual, we identify $\mathcal{X}^*$ with $\mathcal{Y}^*/\mathcal{X}^\bot$.
  Then by assumption, $L$ induces a bounded linear idempotent
  $\widetilde{L}: \mathcal{X}^* \to \mathcal{X}^*$ satisfying
  \begin{equation*}
    \widetilde{L} (\psi \big|_{\mathcal{X}}) = (L \psi) \big|_{\mathcal{X}}, \quad \psi \in \mathcal{Y}^*.
  \end{equation*}
  Since a functional on $\mathcal{X}$ belongs to $\AC_{\mathcal{W}}(\mathcal{X})$
  if and only if it is the restriction to $\mathcal{X}$ of a functional in $\AC_{\mathcal{W}}(\mathcal{Y})$,
  we have $\widetilde{L} \mathcal{X}^* = \AC_{\mathcal{W}}(\mathcal{X})$.
  Finally, if $\varphi \in \mathcal{X}^*$, let $\psi$ be a Hahn--Banach extension of $\varphi$ to $\mathcal{Y}$.
  Then
  \begin{equation*}
    \|\varphi\| = \|\psi\| = \|L \psi\| + \|(I - L) \psi\| \ge \|\widetilde{L} \varphi\|
    + \|(I - \widetilde{L}) \varphi\|.
  \end{equation*}
  The reverse inequality follows from the triangle inequality.
  Hence $\mathcal{X}$ admits a Lebesgue decomposition, which is compatible with that of $\mathcal{Y}$ by construction.
  \end{proof}

We mention here that condition (ii) of the previous result corresponds to a general version of the statement of the classical F. and M. Riesz theorem.
The reader may find it instructive to compare these ideas with \cite[Proposition I.1.16]{HWW}.

\section{Lebesgue decompositions for operator algebras}\label{S:OA}

\subsection{The second dual of an operator algebra}\label{SS:bidualopalg}
Since our focus in the rest of the paper will be on operator algebras, we will often need to work with the operator algebraic structure of biduals. We briefly recall some of the key points below, and refer the reader to \cite{BLM2004} for further details.

Let $\A$ be an operator algebra. The dual space $\A^*$ carries the structure of an $\A^{**}$-bimodule, as follows. Given $b\in \A$ and $\phi\in \A^*$ we define $b \phi\in \A^*$ and $\phi b\in \A^*$ as
\[
(b\phi)(a)=\phi(ab), \quad (\phi b)(a)=\phi(ba)
\]
for every $a\in \A$. Next, given $\phi\in \A^*$ and $\Lambda\in \A^{**}$ we define $\phi \Lambda\in \A^{*}$ and $\Lambda\phi\in \A^{*}$ as
\[
(\phi \Lambda)(a)=\Lambda(a\phi), \quad (\Lambda\phi)(a)=\Lambda(\phi a)
\]
for every $a\in \A$.
In some places, we also write $\Lambda \cdot \varphi = \Lambda \varphi$ and $\varphi \cdot \Lambda = \varphi \Lambda$
to avoid confusion.
Note that these products are compatible with the canonical embedding $\mathcal{A} \hookrightarrow A^{**}$,
in the sense that if $\varphi \in \mathcal{A}^*$ and $a \in \mathcal{A}$, then
$\varphi \widehat{a} = \varphi a$ and $\widehat{a} \varphi = a \varphi$, where $\widehat{a} \in \mathcal{A}^{**}$
denotes the image of $a$ under the canonical embedding.
Finally, given $\Lambda,\Xi\in \A^{**}$ recall that their \emph{Arens products} $\Lambda\cdot_\lambda \Xi\in \A^{**}$ and $\Lambda\cdot_\mu\Xi \in \A^{**}$ are defined as
\[
(\Lambda\cdot _\mu\Xi)(\phi)=\Lambda(\Xi\phi), \quad (\Lambda\cdot _\lambda\Xi)(\phi)=\Xi(\phi \Lambda).
\]
Operator algebras are known to be Arens regular \cite[Corollary 2.5.4]{BLM2004}, so that the two products above agree and we denote them simply by $\Lambda  \Xi$. Moreover, the Banach algebra $\A^{**}$ equipped with the Arens product is completely isometrically isomorphic and weak-$*$ homeomorphic to a weak-$*$ closed subalgebra of Hilbert space operators \cite[Corollary 2.5.6]{BLM2004}, and thus we will view $\A^{**}$ as a weak-$*$ closed operator algebra upon making this identification. It is readily verified that in fact $\A$ is a subalgebra of $\A^{**}$ upon making the usual identifications. If $\A$ happens to be a $\rC^*$-algebra, then $\A^{**}$ is a von Neumann algebra \cite[Paragraph 2.5.3 and Theorem A.5.6]{BLM2004}. 

\subsection{Lebesgue projections}\label{SS:Lebproj}

In our setting of operator algebras, we will assume that the ambient dual space $\mathcal{W}$ is a von Neumann algebra. In particular, this ensures that multiplication is separately weak-$*$ continuous
with respect to the weak-$*$ topology on $\mathcal{W}$.
We start with a basic property of absolutely continuous functionals on operator algebras.

\begin{lemma}\label{L:ideal}
  Let $\W$ be a von Neumann algebra. Let $\A\subset \W$ be an operator algebra. Then, $\ac_\W(\A)^\perp$ is a weak-$*$ closed ideal in $\A^{**}$ and the norm closure of $\ac_\W(\A)$ is a $\mathcal{A}^{**}$-submodule of $\A^*$.
\end{lemma}
\begin{proof}
It is clear that $\ac_\W(\A)^\perp$ is a weak-$*$ closed subspace of $\A^{**}$. Let $\Lambda\in \A^{**}, \Xi\in \ac_\W(\A)^\perp$ and $\phi\in \ac_\W(\A)$. By Goldstine's theorem, we can find a bounded net $(a_\alpha)$ in $\A$ converging to $\Lambda$ in the weak-$*$ topology of $\A^{**}$. Now, because multiplication is separately weak-$*$ continuous, we infer that $a_\alpha\phi,\phi a_\alpha\in \ac_\W(\A)$ and so
$
\Xi(a_\alpha\phi)=\Xi(\phi a_\alpha)=0
$
for every $\alpha$. Consequently,
\[
0=\lim_\alpha \Xi(a_\alpha\phi)=\lim_\alpha (\Xi a_\alpha)(\phi)=(\Xi\Lambda)(\phi)=\Xi(\Lambda \phi)
\]
and
\[
0=\lim_\alpha \Xi(\phi a_\alpha)=\lim_\alpha (a_\alpha \Xi)(\phi)=(\Lambda\Xi)(\phi)=\Xi(\phi\Lambda).
\]
We conclude that $\Xi\Lambda,\Lambda\Xi\in \ac_\W(\A)^\perp$ and $\Lambda\phi,\phi\Lambda\in (\ac_\W(\A)^\perp)_\perp$. Therefore, $\ac_\W(\A)^\perp$ is an ideal while $(\ac_\W(\A)^\perp)_\perp$ is a submodule. Finally, the Hahn--Banach theorem implies that $(\ac_\W(\A)^\perp)_\perp$ is simply the norm closure of $\ac_\W(\A)$.
\end{proof}

Next, we record a standard fact.

\begin{lemma}
  \label{L:proj_lebesgue_dec}
  Let $\mathcal{W}$ be a von Neumann algebra and let $\mathcal{A} \subset \mathcal{W}$ be a unital subalgebra.
  Let $\fz \in \mathcal{A}^{**}$ be a central projection.
  Then 
  \[
    L: \mathcal{A}^* \to \mathcal{A}^*, \quad \varphi \mapsto \fz \cdot \varphi,
  \]
  defines a bounded linear idempotent with
  \begin{equation*}
    \|\varphi\| = \|L \varphi\| + \|(I - L) \varphi\|, \quad \varphi \in \mathcal{A}^{**}.
  \end{equation*}
\end{lemma}

\begin{proof}
  The fact that $L$ is a linear idempotent follows from the fact that the action of $\mathcal{A}^{**}$
  on $\mathcal{A}^*$ is indeed a module action.
  If $\varphi \in \mathcal{A}^*$, then by the Hahn--Banach theorem,
  there exist norm one elements $\Lambda,\Xi \in \mathcal{A}^{**}$ with
  $\Lambda (L \varphi) = \|L \varphi\|$ and $\Xi( (I - L) \varphi) = \|(I - L) \varphi\|$.
  Hence
  \begin{equation*}
    \|L \varphi\| + \|(I - L) \varphi\|
    = \Lambda( \fz \cdot \varphi) + \Xi( (I - \fz) \cdot \varphi)
    = (\Lambda \fz + \Xi (I - \fz)) (\varphi) \le \|\varphi\|,
  \end{equation*}
  where the inequality follows from the fact that $\fz$ is a central projection.
  The reverse inequality is simply the triangle inequality.
\end{proof}

We are now ready to prove the announced result about Lebesgue decomposition
in operator algebras. The statement is reminiscent of the investigations
carried out in  \cite[Section V.3]{HWW}. It is likely known to experts, but we do not have an explicit reference.
\begin{theorem}\label{T:Lebproj}
Let $\W$ be a von Neumann algebra and let $\mathcal{A} \subset \mathcal{W}$ be a unital subalgebra. Then, the following statements are equivalent.
\begin{enumerate}[{\rm (i)}]
  \item The algebra $\mathcal{A}$ admits a Lebesgue decomposition.
\item There exists a central projection $\fz \in \mathcal{A}^{**}$ so that
 $\ac_{\mathcal{W}}(\A) = \fz \mathcal{A}^*$.
\item The space $\ac_{\mathcal{W}}(\mathcal{A})$ is closed and there exists a central projection $\fz \in \mathcal{A}^{**}$ so that $\ac_{\mathcal{W}}(\mathcal{A})^\bot = (I - \fz) \mathcal{A}^{**}$.
\end{enumerate}
In this case, the projections in {\rm (ii)} and {\rm (iii)} agree, and
$\sing_{\mathcal{W}}(\mathcal{A}) = (I - \fz) \mathcal{A}^*$.
Moreover, the Lebesgue decomposition is given by
\begin{equation*}
  \varphi = \fz \cdot \varphi + (I- \fz) \cdot \varphi, \quad \varphi \in \mathcal{A}^*.
\end{equation*}
\end{theorem}

\begin{proof}
  (i) $\Rightarrow$ (iii)
  Since $\mathcal{A}$ admits a Lebesgue decomposition, $\AC_{\mathcal{W}}(\mathcal{A})$
  is an $L$-summand in $\mathcal{A}^*$, so $\AC_{\mathcal{W}}(\mathcal{A})^{\bot}$
  is an $M$-summand (also known as $\ell^\infty$-summand) in $\mathcal{A}^{**}$.
  Every $M$-summand in a unital operator algebra is generated
  by a central projection in the algebra; see for instance \cite[Theorem 4.8.5 2)]{BLM2004}.
  Thus, (iii) holds.

  (iii) $\Rightarrow$ (ii)
  Let $\fz$ be the projection in (iii) and let $\varphi \in \mathcal{A}^*$.
  By the Hahn--Banach theorem, $\varphi \in \ac_{\mathcal{W}}(\mathcal{A})$ if and only if
  $\Lambda(\varphi) = 0$ for all $\Lambda \in \ac_{\mathcal{W}}(\mathcal{A})^\bot$,
  which happens if and only if $(\Lambda (I - \fz)) (\varphi) = 0$
  for all $\Lambda \in \mathcal{A}^{**}$, which in turn is equivalent
  to $(I - \fz) \cdot \varphi = 0$. Hence
  $\AC_{\mathcal{W}}(\mathcal{A}) = \fz \mathcal{A}^*$.

  (ii) $\Rightarrow$ (i) It follows from Lemma \ref{L:proj_lebesgue_dec} that
  the map
  \begin{equation*}
    L: \mathcal{A}^* \to \mathcal{A}^*, \quad \varphi \mapsto \fz \cdot \varphi,
  \end{equation*}
  yields the the Lebesgue decomposition of $\mathcal{A}$.
  In particular, $\SG_{\mathcal{W}}(\mathcal{A}) = (I - L) \mathcal{A}^* = (I - \fz) \mathcal{A}^*$.

  It is clear that the projections in (ii) and (iii) are unique, so the additional statements were already proven above.
\end{proof}

The projection $\fz\in \A^{**}$ in the previous theorem will be referred to as the \emph{Lebesgue projection} of $\A$.

We can now provide the details underlying a statement made in the introduction.

\begin{lemma}\label{L:C*decomp}
 Let $\W$ be a von Neumann algebra and let $\fT\subset \W$ be a weak-$*$ dense unital $\rC^*$-subalgebra. Then, $\fT$ admits a Lebesgue decomposition relative to $\W$.
\end{lemma}
\begin{proof}
First, note that Kaplansky's density theorem \cite[Theorem I.7.3]{davidson1996} implies that the unit ball of $\fT$ is weak-$*$ dense in that of $\W$. We may thus apply Proposition \ref{P:ACclosed} to see that the space $\AC_\W(\fT)$ is norm closed in $\fT^{*}$. Next, Lemma \ref{L:ideal} shows that $\AC_\W(\fT)^\perp$ is a weak-$*$ closed ideal in $\fT^{**}$. An application of \cite[Proposition 1.10.5]{sakai1971} yields the existence of a Lebesgue projection for $\fT$. The desired statement then follows from Theorem \ref{T:Lebproj}.
\end{proof}

For general operator algebras, there is no analogue of the previous result.
Nevertheless, if a Lebesgue decomposition exists, then it has to be compatible with the algebraic structure,
and the decomposition is necessarily implemented by a projection in the second dual, by virtue of Theorem \ref{T:Lebproj}.

We also record the following consequence for the existence of a Lebesgue decomposition on a smaller algebra,
given a Lebesgue decomposition on a larger algebra.

\begin{proposition}
  \label{prop:Lebesgue_proj_smaller_alg}
  Let $\mathcal{W}$ be a von Neumann algebra and let $\mathcal{A} \subset \mathcal{B} \subset \mathcal{W}$
  be unital subalgebras. Assume that $\mathcal{B}$ admits a Lebesgue decomposition
  and let $\fz_{\mathcal{B}} \in \mathcal{B}^{**}$ be the corresponding Lebesgue projection.
  Then, the following assertions are equivalent.
  \begin{enumerate}[\rm (i)]
    \item The algebra $\mathcal{A}$ admits a Lebesgue decomposition that is compatible with that of $\mathcal{B}$.
    \item We have $\fz_{\mathcal{B}} \in \mathcal{A}^{\bot \bot}$.
  \end{enumerate}
\end{proposition}

\begin{proof}
  (i) $\Rightarrow$ (ii)
  Let $\psi \in \mathcal{A}^{\bot}$. Applying Theorem \ref{T:Lebcompat} and Theorem \ref{T:Lebproj},
  we see that $\fz_{\mathcal{B}} \cdot \psi \in \mathcal{A}^{\bot}$.
  In particular, $\fz_{\mathcal{B}}(\psi) = (\fz_{\mathcal{B}} \cdot \psi)(I) = 0$.
  Hence $\fz_{\mathcal{B}} \in \mathcal{A}^{\bot \bot}$.

  (ii) $\Rightarrow$ (i) Let $\fz_{\mathcal{B}} \in \mathcal{A}^{\bot \bot}$, which we identify with $\mathcal{A}^{**}$ in the usual way.
  We claim that $\AC_{\mathcal{W}}(\mathcal{A}) = \fz_{\mathcal{B}} \mathcal{A}^*$.
  Indeed, if $\varphi \in \mathcal{A}^*$ and $\psi \in \mathcal{B}^*$ is any extension of $\varphi$,
  then $\fz_{\mathcal{B}} \varphi = (\fz_{\mathcal{B}} \psi)|_{\mathcal{A}} \in \AC_{\mathcal{W}}(\mathcal{A})$,
since $\fz_{\mathcal{B}} \psi \in \AC_{\mathcal{W}}(\mathcal{B})$.
  Conversely, if $\varphi \in \AC_{\mathcal{W}}(\mathcal{A})$, then
  $\varphi$ extends to an element $\psi \in \AC_{\mathcal{W}}(\mathcal{B})$,
  so $\fz_{\mathcal{B}} \varphi = (\fz_{\mathcal{B}} \psi) |_{\mathcal{A}} = \psi |_{\mathcal{A}} = \varphi$.
  This establishes the claim.

  Since $\AC_{\mathcal{W}}(\mathcal{A}) = \fz_{\mathcal{B}} \mathcal{A}^*$, Lemma \ref{L:proj_lebesgue_dec} implies that multiplication by $\fz_{\mathcal{B}}$
  gives a Lebesgue decomposition of $\mathcal{A}$, which is compatible with that of $\mathcal{B}$ by construction.
\end{proof}

\section{The Gleason--Whitney property}\label{S:GW}

In this section, we make use of a property of inclusions of operator algebras to study the compatibility of Lebesgue decompositions. This sheds some light on the manner in which such decompositions are often constructed in the literature. 

We begin with some motivation.
Let $\Hinf$ denote the algebra of bounded holomorphic functions on the open unit disc $\bD\subset \bC$, equipped with the usual supremum norm over $\mathbb{D}$. Let $\sigma$ denote arclength measure on the unit circle $\bT$. As is well known, $\Hinf$ may be embedded isometrically as a weak-$*$ closed subalgebra of $L^\infty(\bT,\sigma)$, by taking
radial boundary limits. The resulting inclusion $\Hinf\subset L^\infty(\bT,\sigma)$ has the following peculiar property, first discovered by Gleason and Whitney \cite[Theorem 5.4]{gleason1962}.

\begin{theorem}\label{T:classicalGW}
Let $\phi$ be a weak-$*$ continuous functional on $H^\infty(\bD)$. Then, $\phi$ admits a unique Hahn--Banach extension to $L^\infty(\bT,\sigma)$, and this extension must be weak-$*$ continuous.
\end{theorem}

We now wish to isolate this phenomenon. As before, we fix $\W$, the dual space of some normed space. Given a pair of subspaces $\X\subset \Y$ of $\W$, following  \cite{BL2007} and \cite{blecherlabuschagne2007} we say that the inclusion $\X\subset \Y$ has the \emph{Gleason--Whitney property relative to $\W$} if, whenever $\phi\in\ac_\W(\X)$ and $\psi$ is a Hahn--Banach extension of $\phi$ to $\Y$, then $\psi\in \ac_\W(\Y)$. Thus, Theorem \ref{T:classicalGW} says that the inclusion  $\Hinf\subset L^\infty(\bT,\sigma)$ has the Gleason--Whitney property relative to $L^\infty(\bT,\sigma)$. A large class of additional examples is furnished by \cite[Theorem 4.1]{BL2007}. For the purposes of this paper, the analogue of the uniqueness statement in Theorem \ref{T:classicalGW} will not be relevant.

The first step in relating the Gleason--Whitney property to  the compatibility of Lebesgue decompositions is the following;
the proof is very similar to that of \cite[Theorem 5.2]{blecherlabuschagne2007}.

\begin{proposition}\label{P:GWcompatspaces}
Let $\W$ be the dual space of some normed space. Let $\X\subset \Y$ be subspaces of $\W$. Assume that $\X$ and $\Y$ admit compatible Lebesgue decompositions. Then, the inclusion $\X\subset \Y$ has the Gleason--Whitney property relative to $\W$.
\end{proposition}
\begin{proof}
Let $\phi\in \ac_\W(\X)$ and let $\psi\in \Y^*$ be a Hahn--Banach extension of $\phi$.  Write $\psi = \psi_a + \psi_s$ with $\psi_a\in \ac_\W(\Y), \psi_s\in \sing_\W(\Y)$ and $\|\psi\|=\|\psi_a\|+\|\psi_s\|$. Since the Lebesgue decompositions of $\X$ and $\Y$ are compatible, we see that $\psi_s|_\X\in \sing_\W(\X)$. Furthermore, it is always true that $\psi_a|_\X\in \ac_\W(\X)$, so the equality $\phi=\psi_a|_\X+\psi_s|_\X$ and the fact that $\phi\in \ac_\W(\X)$ force $\psi_s|_\X=0$. Hence
  \begin{equation*}
    \|\psi_a\| + \|\psi_s\| = \|\psi\| = \|\phi\| = \|\psi_a |_\X\| \le \|\psi_a\|,
  \end{equation*}
  so that $\psi_s = 0$, and consequently $\psi = \psi_a\in \ac_\W(\Y)$. We conclude that the inclusion $\X\subset \Y$ has the Gleason--Whitney property relative to $\W$.
\end{proof}

Given a Hilbert space $\H$, we denote by $B(\H)$ the $\rC^*$-algebra of bounded linear operators on it, and by $\fK(\H)$ the ideal of compact operators.

We can now identify a class of inclusions of spaces for which the Gleason--Whitney property holds.

\begin{corollary}\label{C:compact}
 Let $\X\subset B(\H)$ be a subspace such that the weak-$*$ closure of $\X\cap \fK(\H)$ in $B(\H)$ contains $\X$. Then, $\X$ and $B(\H)$ admit compatible Lebesgue decompositions relative to $B(\H)$. In particular, the inclusion $\X\subset B(\H)$ has the Gleason--Whitney property relative to $B(\H)$.
\end{corollary}
\begin{proof}
First, we note that $B(\H)$ admits a Lebesgue decomposition by Lemma \ref{L:C*decomp}. Next, let $\phi$ be a continuous linear functional on $B(\H)$ that annihilates $\X$, and write $\phi=\phi_a+\phi_s$ for its Lebesgue decomposition. Arguing as in the proof of \cite[Lemma 2.5]{CTh2021}, we see that $\phi_s$ annihilates $\fK(\H)$, so that $\phi_a=\phi-\phi_s$ annhilates $\X\cap \fK(\H)$. Since $\phi_a$ is absolutely continuous, this means that $\phi_a$ annhilates the weak-$*$ closure of $\X\cap \fK(\H)$, which contains $\X$ by assumption. Thus, $\phi_s$ annihilates $\X$ as well, so we may apply Theorem \ref{T:Lebcompat} to see that $\X$ admits a Lebesgue decomposition compatible with that of $B(\H)$. In particular, the inclusion $\X\subset B(\H)$ has the Gleason--Whitney property relative to $B(\H)$ by virtue of Proposition \ref{P:GWcompatspaces}.
\end{proof}

A concrete example where the previous result is applicable is the space $\Han(\mathcal{H})$ of (little) Hankel operators on a reproducing kernel Hilbert space $\mathcal{H}$ that densely and contractively contains its multiplier algebra $\M(\mathcal{H})$; see e.g.\ \cite[Section 2]{AHM+18} for details.
Thus, we see that $\Han(\mathcal{H})$ admits a Lebesgue decomposition that is compatible with that of $B(\mathcal{H})$. This stands in contrast with the situation for the multiplier algebra $\mathcal{M}(H)$; see Subsection \ref{SS:multipliers} below.

Assume now that $\mathcal{W}$ is a von Neumann algebra, and let $\mathcal{Y} \subset \mathcal{W}$ be a unital subspace.
Recall that a unital contractive linear functional on $\Y$ is called a \emph{state}.  In the original paper \cite{gleason1962},  Theorem \ref{T:classicalGW} is first proved for states, and then generalized to arbitrary functionals. As we will see, this strategy works much more generally. 

Let $\X\subset \Y$ be a unital subspace. We say that the inclusion $\X\subset \Y$ has the \emph{Gleason--Whitney property for states relative to $\W$} if, whenever $\psi$ is a state on $\Y$ with the property that $\psi |_\X \in \ac_\W(\X)$, then we must have $\psi \in \ac_\W(\Y)$.
We show next that this a priori weaker condition is the same as the previous one in the setting of operator algebras, provided that the larger algebra admits a Lebesgue projection, using an argument similar to that used in \cite[Theorem 5.4]{gleason1962}.

\begin{theorem}\label{T:GWstates}
Let $\W$ be a von Neumann algebra. Let $ \B\subset \W$ be unital operator algebra that admits a Lebesgue decomposition. Let $\A\subset \B$ be a unital subalgebra. Assume that the inclusion $\A\subset \B$ has the Gleason--Whitney property for states relative to $\W$. Then, the inclusion $\A\subset \B$ has the usual Gleason--Whitney property relative to $\W$.
\end{theorem}
\begin{proof}
Let $\phi\in\A^*$ be an absolutely continuous functional. To prove the result,  it suffices to assume that $\|\phi\|=1$ and to show that all Hahn--Banach extensions of $\phi$ to $\B$ must be absolutely continuous as well. To see this, let $\psi\in \B^*$ be a Hahn--Banach extension of $\phi$, so that $\|\psi\|=1$. There is $\Lambda\in \A^{**}$ with $\|\Lambda\|=1$ and $\Lambda(\psi)=\Lambda(\phi)=1$. Note then that the product $\Lambda \psi$ is a state on $\B$ which extends the state $\Lambda \phi$ on $\A$. By assumption, we conclude that $\Lambda\psi\in \ac_\W(\B)$.

Next, write $\psi=\psi_a+\psi_s$ where $\psi_a\in \ac_\W(\B), \psi_s\in \sing_\W(\B)$ and $1=\|\psi_a\|+\|\psi_s\|$. We infer that $\Lambda(\psi_s)=\|\psi_s\|$. By Theorem \ref{T:Lebproj}, we see that $\SG_\W(\B)$ is a submodule of $\B^*$. Moreover, the fact that $\B$ admits a Lebesgue decomposition forces $\ac_\W(\B)$ to be norm-closed, whence Lemma \ref{L:ideal} yields that $\ac_\W(\B)$ is a submodule of $\B^*$ as well. Therefore, we have that $\Lambda\psi_a\in \ac_\W(\B)$ and $\Lambda \psi_s\in \sing_\W(\B)$. Since
$
\Lambda \psi=\Lambda\psi_a+\Lambda \psi_s
$
we infer that $\Lambda \psi_s=\Lambda\psi-\Lambda\psi_a\in \ac_\W(\B)\cap\sing_\W(\B)$ so that $\Lambda\psi_s=0$. In particular, we find
\[
\|\psi_s\|=\Lambda(\psi_s)=(\Lambda \psi_s)(I)=0
\]
and thus $\psi=\psi_a\in \ac_\W(\B)$.
\end{proof}

We record an easy consequence.

\begin{corollary}\label{C:GWtriple}
Let $\W$ be a von Neumann algebra. Let $ \B\subset \W$ be unital operator algebra that admits a Lebesgue decomposition. Let $\A\subset \B$ be a unital subalgebra. Assume that the inclusion $\A\subset \B$ has the Gleason--Whitney property relative to $\W$. If $\C\subset\B$ is another unital subalgebra containing $\A$, then the inclusions $\A\subset \C$ and $\C\subset \B$ both have the Gleason--Whitney property relative to $\W$.
\end{corollary}
\begin{proof}
The fact that  $\A\subset \C$ has the Gleason--Whitney property relative to $\W$ is an immediate consequence of the fact that $\A\subset \B$ has it. Next, let $\psi$ be a state on $\B$ such that $\psi|_\C$ is absolutely continuous. In particular, $\psi|_\A\in \ac_\W(\A)$. Since $\A$ is unital, we see that $\psi$ is Hahn--Banach extension of $\psi|_\A$ to $\B$, so by assumption we must have that $\psi$ is absolutely continuous. This shows that $\C\subset \B$ has the Gleason--Whitney property for states relative to $\W$, and thus the usual property by Theorem \ref{T:GWstates}.
\end{proof}

We focus now on unital subalgebras of $B(\H)$ and we identify some necessary condition for the Gleason--Whitney property to hold. The following fact is elementary, but it will be very useful in Section \ref{S:examples}.

\begin{lemma}\label{L:GWquotient}
Let $\A \subset B(\H)$ be a unital subalgebra, and let $\fT\subset B(\H)$ be a unital $\rC^*$-algebra containing both $\A$ and $\fK(\H)$.
Assume that there exists an absolutely continuous state $\varphi$ on $\mathcal{A}$ satisfying
\[
|\phi(a)|\leq \|a+\fK(\H)\|, \quad a\in \A.
\]
Then the inclusion $\mathcal{A} \subset \fT$ does not have the Gleason--Whitney property for states relative
to $B(\mathcal{H})$.
 \end{lemma}

\begin{proof}
Denoting by $q:\fT\to \fT/\fK(\H)$ the natural quotient map, we find a state $\phi':q(\A)\to \bC$ such that $\phi=\phi'\circ q$. Let $\psi$ be a Hahn--Banach extension of $\phi'$ to $\fT/\fK(\H)$. Then, $\psi\circ q$ is a state on $\fT$ extending $\phi$, which cannot be absolutely continuous as it annihilates the weak-$*$ dense set $\fK(\H)$. Hence, the inclusion $\A\subset \fT$ does not have the Gleason--Whitney property for states relative to $B(\H)$.
  \end{proof}

We also record a general result that may be of independent interest. If $\A\subset B(\H)$ is a unital subalgebra, we denote by $\rC^*(\A)\subset B(\H)$ the $\rC^*$-algebra that it generated. A unital $*$-homomorphism $\pi:\rC^*(\A)\to B(\E)$ is said to have the \emph{unique extension property with respect to $\A$} if, whenever $\psi:\rC^*(\A)\to B(\E)$ is a unital completely positive map agreeing with $\pi$ on $\A$, we necessarily have that $\psi$ and $\pi$ agree on $\rC^*(\A)$. The reader may consult \cite{arveson1969},\cite{dritschel2005} or \cite{davidson2015} for more background on this topic.

\begin{theorem}\label{T:GWBH}
Let $\A \subset B(\H)$ be a unital subalgebra such that $\rC^*(\A)$ contains the ideal of compact operators. Assume that the inclusion $\A\subset \rC^*(\A)$ has the Gleason--Whitney property for states relative to $B(\H)$. Then, the identity representation of $\rC^*(\A)$ has the unique extension property with respect to $\A$.
  \end{theorem}

  \begin{proof}
Assume that the identity representation of $\rC^*(\A)$ does not have the unique extension property with respect to $\A$.  An application of Arveson's boundary theorem \cite[Theorem 2.1.1]{arveson1972} then reveals that the quotient map $q:\rC^*(\A)\to \rC^*(\A)/\fK(\H)$ is completely isometric on $\A$. Therefore, Lemma \ref{L:GWquotient}, applied to any absolutely continuous state on $\mathcal{A}$, implies that the inclusion $\A\subset \rC^*(\A)$ cannot have the Gleason--Whitney property for states relative to $B(\H)$.
  \end{proof}

We now go back to our goal of relating the Gleason--Whitney property of an inclusion to the compatibility of Lebesgue decompositions. The following is one of our main results, and gives a sort of converse to Proposition \ref{P:GWcompatspaces}.

\begin{theorem}\label{T:GWcompat}
Let $\W$ be a von Neumann algebra. Let $\B\subset \W$ be a unital operator algebra and let $\A\subset \B$ be a unital subalgebra. Assume that both $\A$ and $\B$ admit Lebesgue decompositions, and let $\fz_\A$ and $\fz_\B$ be the corresponding Lebesgue projections.
Then, the following statements are equivalent.
\begin{enumerate}[{\rm (i)}]
\item The inclusion $\A\subset \B$ has the Gleason--Whitney property relative to $\W$.
\item We have $\fz_\A=\fz_\B$.
\item The Lebesgue decompositions of $\A$ and $\B$ are compatible.
\end{enumerate}
\end{theorem}
\begin{proof}
Before starting the proof, we recall that we identify $\A^{**}$ with a unital subalgebra of $\B^{**}$. More precisely, an element $\Lambda\in \A^{**}$ is identified with $\Lambda'\in \B^{**}$ where
\[
\Lambda'(\psi)=\Lambda(\psi|_\A), \quad \psi\in \B^*.
\]

(i) $\Rightarrow$ (ii):  It is trivial that if $\phi\in \ac_\W(\B)$, then $\phi|_\A\in \ac_\W(\A)$. This readily implies that if $\Lambda\in \A^{**}$ lies in $(\ac_\W(\A))^\perp$, then we also have $\Lambda\in(\ac_\W(\B))^\perp$. By virtue of Theorem \ref{T:Lebproj}, we infer that $(I-\fz_\A)\A^{**}\subset (I-\fz_\B)\B^{**}$, and thus $I-\fz_\A\leq I-\fz_\B$.

Conversely, we claim that $\fz_\A\leq \fz_\B$, or that $\fz_\A(I-\fz_\B)=0$. 
To see this, we need to establish that if $\psi$ is a state on $\W$, then $(\fz_\A(I-\fz_\B))(\psi)=0.$ If $\fz_A(\psi)=0$, then the fact that $\fz_\B$ is central in $\B^{**}$ yields
\begin{align*}
 (\fz_\A(I-\fz_\B))(\psi)&=\widehat\psi(\fz_\A(I-\fz_\B))=\widehat\psi(\fz_\A(I-\fz_\B)\fz_\A)\\
 &\leq \widehat\psi(\fz_\A)=\fz_\A(\psi)=0.
\end{align*}
Assume henceforth that $\psi(\fz_\A)\neq 0$ and put $\phi=\fz_\A \cdot \psi \cdot \fz_\A$. Clearly, $\phi$ is a positive linear functional on $\W$ and $\phi(I)=\fz_\A(\psi)\neq 0$. Thus, the functional $\theta=\frac{1}{\phi(I)}\phi$ is a state on $\W$. Theorem \ref{T:Lebproj} implies that $\phi|_\A = \fz_\A (\psi |_{\A}) \fz_\A \in \ac_\W(\A)$, so that $\theta|_\A\in \ac_\W(\A)$ as well.
By virtue of the Gleason--Whitney property for states, we infer that $\theta|_\B\in \ac_\W(\B)$ and therefore $\phi|_\B=\fz_\A\cdot \psi|_\B \cdot\fz_\A\in \ac_\W(\B)$. Thus,  
\begin{align*}
 (\fz_\A(I-\fz_\B))(\psi)&=(\fz_\A(I-\fz_\B)\fz_\A)(\psi)=(I-\fz_\B)  (\fz_\A\cdot \psi|_\B\cdot \fz_\A)= 0
\end{align*}
 since $(I-\fz_\B)\B^{**}=(\ac_\W(\B))^\perp$.  This establishes the claim, so $\fz_\A\leq \fz_\B$, which, when combined with the previous paragraph, yields $\fz_\A=\fz_\B$.

(ii) $\Rightarrow$ (iii): This follows
from Proposition \ref{prop:Lebesgue_proj_smaller_alg}.

(iii) $\Rightarrow$ (i): This follows directly from Proposition \ref{P:GWcompatspaces}.
\end{proof}

The previous argument relies heavily on the existence of Lebesgue projections. It is conceivable that statements (i) and (iii) in Theorem \ref{T:GWcompat} are equivalent more generally for inclusions of normed spaces where there are no Lebesgue projections to exploit, but we were not able to resolve this question.

We also point here that the previous result sheds some light on the classical inclusion $\Hinf\subset L^\infty(\bT,\sigma)$. The algebra $\Hinf$ is shown in \cite{ando1978} to admit a Lebesgue projection, by utilizing the corresponding projection of the von Neumann algebra $L^\infty(\bT,\sigma)$. (Alternative proofs
can be found in \cite{HWW}, see especially Example IV.1.1 (d) and Remark III 1.8 there.)
Theorem \ref{T:GWcompat} offers an abstract framework explaining why this approach works: the key point is the classical Gleason--Whitney theorem (Theorem \ref{T:classicalGW}).

\section{Constructing projections from the Gleason--Whitney property}\label{S:obstruction}
In this section, we analyze inclusions $\A\subset\B$ of unital operator algebras, where $\B$ admits a Lebesgue projection. Assuming that the inclusion has the Gleason--Whitney property, we construct a projection in $\mathcal{A}^{**}$ that behaves like the Lebesgue projection of $\B$, at least as far as the absolutely continuous characters of $\A$ are concerned.

Before proceeding, we recall some notions from Akemann's non-commutative topology \cite{akemann1969},\cite{akemann1970left}. Let $\fT$ be a $\rC^*$-algebra and let $p\in \fT^{**}$ be a projection. Then, $p$ is said to be \emph{open} if there is an increasing net of positive contractions $(t_\alpha)$ in $\fT$ converging to $p$ in the weak-$*$ topology of $\fT^{**}$. Equivalently, this means that there is a closed left ideal $\fJ\subset \fT$ 
with the property that $\fJ^{\perp\perp}=\fT^{**}p$. For our purposes, we will be concerned with Lebesgue projections that happen to be open. The next result identifies a sufficient condition for this to hold.

\begin{proposition}\label{P:openLeb}
Let $\fT\subset B(\H)$ be a unital $\rC^*$-algebra containing the ideal of compact operators on $\H$. Then, $\fT$ admits a non-trivial open Lebesgue projection.
\end{proposition}
\begin{proof}
Since $\fT$ is a $\rC^*$-algebra containing the ideal $\fK$ of compact operators, it must be weak-$*$ dense in $B(\H)$. By virtue of Lemma \ref{L:C*decomp}, we infer that $\fT$ admits a Lebesgue projection, which we denote by $\fz$.  Arguing as in the proof of \cite[Lemma 2.5]{CTh2021}, we obtain that $\fK^{\perp\perp}=\fz \fT^{**}$, whence  $\fz$ is open. Since $\fT$ is unital, there exists a state on $\fT$ annihilating $\fK$, and this state cannot be absolutely continuous, so $\fz\neq I$.
\end{proof}

Our main technical tool is the following, which shows how to leverage the Gleason--Whitney property.

\begin{lemma}
  \label{L:approxGW}
  Let $\W$ be a von Neumann algebra.
  Let $\fT\subset\W$ be a unital $\rC^*$-algebra  and let $\A\subset \fT$ be a unital subalgebra. Assume that $\fT$ admits an open Lebesgue projection $\fz$, and that the inclusion $\A\subset \fT$ has the Gleason--Whitney property relative to $\W$.  Let $\phi$ be an absolutely continuous state on $\A$.
  Then,
  \begin{equation*}
    \sup \{ \re \varphi(a): a \in \mathcal{A} \text{ with } \re a \le \fz-I\} = 0.
  \end{equation*}
\end{lemma}

\begin{proof}
  Let $c$ denote the supremum in the statement.
  If $\re a \le \fz - I$, then in particular $\re a \le 0$ and so $\re \varphi(a) \le 0$.
  Thus, we have to show that $c \ge 0$.

 By assumption, there is a net of positive contractions $(t_\alpha)$ in $\fT$ that increase to $\fz$ in the weak-$*$ topology of $\fT^{**}$.
 Let $\E_\phi$ denote the weak-$*$ compact  subset of states on $\fT$ that extend $\phi$.
 For each $\alpha$, we find that
 \begin{equation}
   \label{eqn:GW_approx}
   \begin{split}
     c &\ge \sup\{\re \phi(a):a\in \A \text{ such that }\re a\leq t_\alpha -I \} \\
       &=\min \{ \Phi(t_\alpha):\Phi\in \E_\phi\}-1;
   \end{split}
 \end{equation}
 see \cite[Proposition 6.2]{arveson2011} for the equality.

Since $\A\subset \fT$ has the Gleason--Whitney property, it follows that $\E_\phi$ consists entirely of absolutely continuous states on $\fT$. In particular, we see that
\[
 \lim_{\alpha}\Phi(t_\alpha)=\fz(\Phi)=1, \quad \Phi\in \E_\phi.
\]
By Dini's theorem, we conclude that the increasing net $(t_\alpha)$ converges uniformly to $1$ on $\E_\phi$, so that
\[
 \lim_{\alpha} \min \{ \Phi(t_\alpha):\Phi\in \E_\phi\}=1.
\]
In combination with \eqref{eqn:GW_approx}, we conclude that $c \ge 0$.
\end{proof}

The basic idea behind our construction is to extract a weak-$*$ limit of a sequence approximating the supremum above.
To ensure boundedness, we make use of the exponential function, and require the following elementary fact about norms of exponentials.

\begin{lemma}
  \label{L:vNI}
  Let $\fT$ be a unital $C^*$-algebra and let $a \in \fT$.
  \begin{enumerate}[\rm (a)]
    \item If $\re a \le 0$, then $\|e^a\| \le 1$.
    \item If $q \in \fT$ is a projection commuting with $a$ such that $\re a \le -q$,
      then $\|q e^a\| \le e^{-1}$.
  \end{enumerate}
\end{lemma}

\begin{proof}
  We may assume that $\fT \subset B(\mathcal{H})$ for some Hilbert space $\mathcal{H}$.
  Then part (a) is an immediate consequence of the half plane version of von Neumann's inequality,
  which follows from the disc version with the help of the Cayley transform;
  see \cite[Section 5.3, p.276]{vN1951}.

  For the proof of (b), notice that since $q$ commutes with $a$, we may replace $\mathcal{H}$ with $q \mathcal{H}$
  to achieve that $q=1$, in which the case the result again follows from the half plane version
  of von Neumann's inequality.
\end{proof}

We are now ready to prove the main result of this section.

\begin{theorem}\label{T:Lebesgue_projection_character}
Let $\W$ be a von Neumann algebra. Let $\fT\subset\W$ be a unital $\rC^*$-algebra  and let $\A\subset \fT$ be a unital subalgebra. Assume that $\fT$ admits an open Lebesgue projection $\fz$, and that the inclusion $\A\subset \fT$ has the Gleason--Whitney property relative to $\W$.  Let $\chi$ be an absolutely continuous character on $\A$. Then, there is a non-zero projection $p\in \A^{**}$ such that $p\leq \fz$ and $p(\chi)=1$.
\end{theorem}

\begin{proof}
  Let
  \begin{equation*}
    S = \{ \Lambda \in \mathcal{A}^{**}: \|\Lambda\| \le 1, \|(I - \fz) \Lambda\| \le e^{-1} \text{ and } \Lambda(\chi) = 1 \}.
  \end{equation*}
  We first claim that $S \neq \emptyset$.
  To see this, we apply Lemma \ref{L:approxGW} to obtain a sequence $(a_n)$ in $\mathcal{A}$
  with $\re a_n \le \fz - I$ for all $n \in \mathbb{N}$ and
  $\lim_{n \to \infty} \re \chi(a_n) = 0$.
  Let $b_n = \exp(a_n)$.
  Then Lemma \ref{L:vNI} shows that
  $\|b_n\| \le 1$ and $\|(I - \fz) b_n\| \le e^{-1}$ for all $n \in \mathbb{N}$.
  Moreover, $\lim_{n \to \infty} |\chi(b_n)| = 1$.
  By replacing each $b_n$ with a suitable unimodular multiple of $b_n$, we may
  achieve that $\lim_{n \to \infty} \chi(b_n) = 1$.
  Thus, any weak-$*$ cluster point of $(b_n)$ in $\mathcal{A}^{**}$ belongs to $S$.

  Next, notice that $S$ is closed under multiplication since $\chi$ is a character.
  Moreover, $S$ is weak-$*$ compact.
  In this setting, \cite[Theorem 1.1]{FK1989} implies that $S$ contains an idempotent $p$,
  which has to be a projection since $\|p\| \le 1$.
  By definition of $S$, we see that $p(\chi) = 1$.
  Finally, using  that $\fz$ is central, it follows
  that $(I - \fz) p$ is an idempotent of norm strictly less than one,
  and hence $(I - \fz) p = 0$. Thus, $p \le \fz$.
\end{proof}

We single out the following immediate consequence.

\begin{corollary}\label{C:obstructionK}
Let $\fT\subset B(\H)$ be a unital $\rC^*$-algebra containing the compact operators. Let $\A\subset \fT$ be a unital subalgebra such that the inclusion $\A\subset \fT$ has the Gleason--Whitney property relative to $B(\H)$. If $\A$ admits an absolutely continuous character, then $\A^{**}$ contains a non-trivial projection.
\end{corollary}
\begin{proof}
Combine Proposition \ref{P:openLeb} and Theorem \ref{T:Lebesgue_projection_character}.
\end{proof}

We interpret the previous result as exhibiting an obstruction for an inclusion to satisfy the Gleason--Whitney property. We close this section with an example illustrating this obstruction.

\begin{example}\label{E:nilpotent}
Let $T\in B(\H)$ be a non-zero nilpotent operator, so that there is a positive integer $m \geq  2$ with $T^m=0$. Let $\A\subset B(\H)$ denote the unital subalgebra generated by $T$ . Let $v\in \ker T$ be a unit vector and let $\chi:\A\to\bC$ be defined as
\[
\chi(a)=\langle av,v\rangle, \quad a\in \A.
\]
Clearly, $\chi$ is an absolutely continuous character. 

If $a \in \mathcal{A}$, then $a = p(T)$ for some polynomial $p$ and so the spectral mapping theorem shows that $\sigma(a) = \{p(0)\}$ is a singleton.
This shows that $\mathcal{A}$ does not contain any non-trivial projections.
Because $\A$ is finite-dimensional, it follows that $\A^{**}$ contains no non-trivial projections either. In light of Corollary \ref{C:obstructionK}, we infer that the inclusion $\A\subset \fT$ cannot have the Gleason--Whitney property relative to $B(\H)$ for any unital $\rC^*$-algebra $\fT\subset B(\H)$ containing the compact operators.
\qed
\end{example}

\section{Examples}\label{S:examples}

\subsection{Multiplier algebras on the ball}\label{SS:multipliers} 
In this subsection, we examine various examples of inclusions of operator algebras arising naturally from a class of spaces of holomorphic functions on the ball. We briefly recall the details below, and refer the reader to \cite{DH2020} and to the works cited therein for additional details.

Fix an integer $d\geq 1$. Let $\bB_d\subset \bC^d$ denote the open unit ball and let $\bS_d$ be the unit sphere. Let $\H$ be a regular, unitarily invariant reproducing kernel Hilbert space on $\bB_d$, meaning that the reproducing kernel of $\mathcal{F}$ has the form
\begin{equation*}
  K(z,w) = \sum_{n=0}^\infty a_n \langle z,w \rangle^n
\end{equation*}
for some sequence $(a_n)$ of strictly positive numbers satisfying $\lim_{n \to \infty} \frac{a_n}{a_{n+1}} = 1$
and $a_0 = 1$.
The monomials $$\{z_1^{m_1}z_2^{m_2}\ldots z_d^{m_d}: m_j\geq 0\}$$ form an orthogonal basis for $\H$.  Important examples of such spaces include the Drury--Arveson space and the Hardy space in any dimension, along with the Dirichlet space.

We denote by $\M(\H)$ the multiplier algebra of $\H$, which always contains the polynomials. Accordingly, we let $\A(\H)\subset \M(\H)$ denote the norm-closure of the polynomials.  We write $\fT(\H)=\rC^*(\A(\H))$. For $f\in \M(\H)$, we write $M_f:\H\to \H$ for the corresponding multiplication operator,
and we identify $f$ with $M_f$.

It follows from \cite[Theorem 4.6]{GHX04} that this $\rC^*$-algebra contains the ideal $\fK(\H)$ of compact operators, and that the quotient $\fT(\H)/\fK(\H)$ is $*$-isomorphic to $\rC(\bS_d)$ via a map such that
 \[
 M_{z_j}+\fK(\H)\mapsto z_j|_{\bS_d}, \quad 1\leq j\leq d.
  \]

\begin{proposition}\label{P:MFTF}
The inclusions $\A(\H)\subset \fT(\H)$ and $\M(\H)\subset \rC^*(\M(\H))$ do not have the Gleason--Whitney property for states relative to $B(\H)$.
\end{proposition}
\begin{proof}
We define a state $\phi$ on $\M(\H)$ as
\[
\phi(M_f)=f(0),\quad f\in \M(\H).
\]
For each positive integer $m$, we let $e_m=z_1^m/\|z_1^m\|$. It is easily verified that
\[
\langle fe_m,e_m \rangle=f(0), \quad f\in \M(\H)
\]
which shows that $\phi$ is absolutely continuous. 

Next, the orthonormal sequence $(e_m)$ converges to $0$ in the weak topology of $\H$.
Thus, if $K$ is a compact operator on $\H$, we find
\[
f(0)=\lim_{m\to\infty}\langle (M_f+K)e_m,e_m\rangle, \quad f\in \M(\H).
\]
We then see that
\[
|\phi(M_f)|\leq \|M_f+\fK(\H)\|, \quad f\in \M(\H).
\]
Lemma \ref{L:GWquotient} then implies that the inclusions $\A(\H)\subset \fT(\H)$ and $\M(\H)\subset \rC^*(\M(\H))$ do not have the Gleason--Whitney property for states relative to $B(\H)$. 
\end{proof}

In the last result, we note that $\A(\H)$ admits a Lebesgue decomposition by  \cite[Theorem 3.2]{DH2020}. Being a $\rC^*$-algebra containing the compacts, so does $\fT(\H)$ by Lemma \ref{L:C*decomp}.
Hence, Theorem \ref{T:GWcompat} implies that the Lebesgue decompositions of $\A(\H)$ and of $\fT(\H)$ are not compatible.

Functionals on $\fT(\H)$ with absolutely continuous restrictions to $\A(\H)$ are  sometimes also called \emph{Henkin} and have been the object of recent study \cite{CT2022ncHenkin},\cite{CMT2022}.

\subsection{Uniform algebras}\label{SS:unifalg}

The most classical example of an inclusion of operator algebras with the Gleason--Whitney property is $\Hinf\subset L^\infty(\bT,\sigma)$ (see Theorem \ref{T:classicalGW}). As such, it is natural to seek further examples among the class of \emph{uniform algebras}: unital subalgebras of commutative $\rC^*$-algebras. The purpose of this subsection is to explore this possibility. 

First, we show how the classical theorem can be used to identify other inclusions with the Gleason--Whitney property.

\begin{example}\label{E:douglasalg}
Let $\A\subset  L^\infty(\bT,\sigma)$ be a subalgebra containing $\Hinf$; this is usually summarized by saying that $\A$ is a  \emph{Douglas algebra} \cite{sarason1973}. Theorem \ref{T:classicalGW} along with Corollary \ref{C:GWtriple} implies that the inclusions $H^\infty(\bD)\subset \A$ and $\A\subset L^\infty(\bT,\sigma)$ both have the Gleason--Whitney property relative to $L^\infty(\bT,\sigma)$. 
\qed
\end{example}

The rest of this subsection will illustrate that few instances of this phenomenon can be found among other standard examples.

For each $d\geq 1$, we denote by $\rA(\bB_d)$ the \emph{ball algebra}, that is the norm closure of the polynomials in $\rC(\bS_d)$. Alternatively, $\AB$ is the closed unital subalgebra of $\rC(\bS_d)$ consisting of functions that extend holomorphically to $\bB_d$.

\begin{proposition}\label{P:ballalgebra}
Let $\sigma_d$ denote the unique, rotation invariant, regular Borel probability measure on $\bS_d$.
Then, the inclusion $\AB\subset \rC(\bS_d)$ has the Gleason--Whitney property relative to $L^\infty(\bS_d,\sigma_d)$ if and only if $d=1$.
\end{proposition}
\begin{proof}
Assume first that $d=1$. Let $\phi$ be an absolutely continuous functional on $\AD$. By the F. and M. Riesz theorem \cite[page 47]{hoffman1988}, we see that any extension of $\phi$ to $\rC(\bT)$ must be absolutely continuous as well. Thus, the inclusion $\AD\subset \rC(\bT)$ has the Gleason--Whitney property relative to $L^\infty(\bT,\sigma_1)$.

Assume now that $d>1$.  Let $\psi$ be the state on $\rC(\bS_d)$ defined as
\[
  \psi(f)=\int_{\bT}f(\zeta,0,\ldots,0)d\sigma_1(\zeta), \quad f\in \rC(\bS_d).
\]
Cauchy's formula implies that
\[
\psi(f)=f(0)=\int_{\bS_d}fd\sigma_d, \quad f\in \AB
\]
and as such the restriction of $\psi$ to $\AB$ is absolutely continuous relative to $L^\infty(\bS_d,\sigma_d)$. On the other hand, when $d>1$ it is clear that $\sigma_1$ is not absolutely continuous with respect to $\sigma_d$ (in the sense of measures), so $\psi$ is not an absolutely continuous state on $\rC(\bS_d)$. We conclude that the inclusion $\AB\subset \rC(\bS_d)$ does not have the Gleason--Whitney property relative to $L^\infty(\bS_d,\sigma_d)$. 
\end{proof}

To analyze further examples, the following general principle will be useful.
\begin{proposition}\label{P:unifalgGW}
  Let $\A$ be a uniform algebra and let $\pi:\A\to B(\H)$ be a unital completely isometric homomorphism. Let $\fT\subset B(\H)$ be a unital $\rC^*$-algebra containing both $\pi(\A)$ and the ideal of compact operators on $\H$. If $\rC^*(\pi(\mathcal{A}))$ either contains no non-zero compact operator or is irreducible, then  the inclusion $\pi(\A)\subset \fT$ does not have the Gleason--Whitney property for states relative to $B(\H)$. 
\end{proposition}
\begin{proof}
  Assume first that $\rC^*(\pi(\A))$ contains no non-zero compact operator. Let $q:\fT\to \fT/\fK(\H)$ denote the natural quotient map. The $*$-homomorphism $q|_{\rC^*(\pi(\A))}$ is injective, and hence completely isometric.  Applying Lemma \ref{L:GWquotient} to any absolutely continuous state on $\pi(\mathcal{A})$ implies that the inclusion $\pi(\A)\subset \fT$ does not have the Gleason--Whitney property for states relative to $B(\H)$.

 The remaining case is when the intersection $\rC^*(\pi(\A))\cap \fK(\H)$ is non-trivial and $\rC^*(\pi(\mathcal{A}))$ is irreducible. Then $\fK(\H)\subset \rC^*(\pi(\A))$ \cite[Corollary I.10.4]{davidson1996} so in particular $\rC^*(\pi(\A))$ is not commutative. Recall now that since $\A$ is a uniform algebra, there is a commutative $\rC^*$-algebra $\C$ containing $\A$.
 Assume that the inclusion $\pi(\A)\subset \rC^*(\pi(\A))$ has the Gleason--Whitney property for states relative to $B(\H)$. Invoking Theorem \ref{T:GWBH}, we see that the identity representation of $\rC^*(\pi(\A))$ has the unique extension property with respect to $\pi(\A)$. This implies that $\rC^*(\pi(\A))$ is the $\rC^*$-envelope of $\A$,
 and hence is a quotient of $\C$; see \cite[Theorem 4.1]{dritschel2005}.
 In particular, $\rC^*(\pi(\A))$ is commutative, which is absurd. We conclude that the inclusion $\A\subset \rC^*(\pi(\A))$ does not have the Gleason--Whitney property for states relative to $B(\H)$, so that neither does $\A\subset \fT$.
\end{proof}

\begin{example}
  \label{E:Linfty}
  Let $\pi:L^\infty([0,1]) \to B(L^2([0,1]))$ be the usual representation of $L^\infty([0,1])$ as multiplication operators. Since no non-zero multiplication operator is compact, Proposition \ref{P:unifalgGW} shows that the inclusion $\pi(L^\infty([0,1])) \subset B(L^2([0,1]))$ does not have the Gleason--Whitney property for states relative to $B(L^2([0,1]))$.
\end{example}

The next example shows how to use Proposition \ref{P:unifalgGW} to recover the conclusion of Proposition \ref{P:MFTF}, at least for the Hardy space on the ball.

\begin{example}\label{E:discalgBH}
  Let $\H = H^2(\mathbb{B}_d)$ be the Hardy space on the ball; this is an instance of the class of spaces considered in Subsection \ref{SS:multipliers}. In particular, $\rC^*(\A(\H))$ contains the ideal of compact operators on $\H$, and hence is irreducible.
  It is well known that $A(\mathcal{H}) = A(\mathbb{B}_d)$ is a uniform algebra.
  Applying Proposition \ref{P:unifalgGW} to the map $\pi: A(\mathcal{H}) \to B(\mathcal{H}), f \mapsto M_f$,
  and identifying $A(\mathcal{H})$ with its image under $\pi$ as before, it
  follows that the inclusion $\A(\H) \subset \fT(\H)$ does not have the Gleason--Whitney property for states relative to $B(\H)$. 
\qed
\end{example}

The next development requires the following technical tool, inspired by  an idea from \cite{sarason1973}.
 
\begin{lemma}\label{L:sarason}
Let $X$ be a compact Hausdorff space and let $\mu$ be a regular Borel probability measure on $X$.
Assume that $K \subset X$ is a closed, nowhere dense subset with $\mu(K) > 0$.
Let $\phi$ be a state on $\rC(X)$.  Then, there is a state $\psi$ on $L^\infty(X,\mu)$ extending $\phi$ with the property that 
\[
     \psi (( 1 - \chi_K) f) = \psi(f), \quad f\in L^\infty(X,\mu)
    \]
    where $\chi_K$ denotes the characteristic function of $K$.
\end{lemma}
\begin{proof}
 Consider the subspace $V \subset L^\infty(X,\mu)$ consisting of all functions
  of the form $g \chi_K + h$ with $g, h \in \rC(X)$. Given $g,h\in \rC(X)$, note that $\|g\chi_K+h\|$ dominates the essential supremum of $h$ over $X\setminus K$. This latter set is dense by choice of $K$, so by continuity of $h$ we obtain that 
  $\|h\| \le
  \|g \chi_K + h\|$. We may thus define a unital contractive linear functional $\psi_0$ on $V$ as
  \begin{equation*}
    \psi_0(g \chi_K + h) = \phi(h), \quad g,h\in \rC(X).
  \end{equation*}
  Choose a Hahn--Banach extension $\psi$ of $\psi_0$ to $L^\infty(X,\mu)$. Then, $\psi$ is a state extending $\phi$ and $\psi (1 - \chi_K) = \phi(1)= 1$. Since $1-\chi_K$ is a projection in $L^\infty(X,\mu)$, it follows, for instance from a multiplicative domain argument (see \cite[Theorem 3.18]{paulsen2002}) that
  \begin{equation*}
    \psi (( 1 - \chi_K) f) = \psi(f), \quad f\in L^\infty(\bT,\sigma). \qedhere
  \end{equation*}
 \end{proof}

We now give two applications of this lemma.

\begin{proposition}\label{P:C(X)GW}
Let $X$ be a compact Hausdorff space and let $\mu$ be a regular Borel probability measure on $X$.
Assume that there exists a closed, nowhere dense subset $K \subset X$ with $\mu(K) > 0$.
Then, the inclusion $\rC(X)\subset L^\infty(X,\mu)$ does not have the Gleason--Whitney property for states relative to $L^\infty(X,\mu)$.
\end{proposition}
\begin{proof}
Let $\theta$ be the absolutely continuous state on $L^\infty(X,\mu)$ defined as
\[
\theta(f)=\int_X fd\mu,\quad f\in L^\infty(X,\mu).
\]
Define $\phi=\theta|_{\rC(X)}$. Apply Lemma \ref{L:sarason} to find a state  $\psi$ on $L^\infty(X,\mu)$ extending $\phi$ with the property that 
\[
     \psi (( 1 - \chi_K) f) = \psi(f), \quad f\in L^\infty(X,\mu).
    \]
 Since $\rC(X)$ is weak-$*$ dense in $L^\infty(X,\mu)$, if $\psi$ were absolutely continuous then it would necessarily coincide with $\theta$. However, this is not the case as $\psi(\chi_K)=0$ while $\theta(\chi_K)=\mu(K)>0$.
\end{proof}

With a bit more care, Lemma \ref{L:sarason} can also be used to show the following.
This result is also contained, at least implicitly, in the work of Miller--Olin--Thomson \cite{MOT86}; see Remark \ref{rem:MOT-GW} below.

\begin{proposition}\label{P:GWADHinf}
The inclusion $\AD\subset \Hinf$ does not have the Gleason--Whitney property for states relative to $L^\infty(\bT,\sigma)$.
\end{proposition}
\begin{proof}
Let $\theta$ be the state on $\Hinf$ of evaluation at the origin, and let $\phi=\theta|_{\AD}$. It follows easily from Cauchy's formula that
\[
\theta(f)=\int_{\bT}fd\sigma, \quad f\in \Hinf.
\]
In particular, both $\theta$ and $\phi$ are absolutely continuous states. If the inclusion $\AD\subset \Hinf$ had the Gleason--Whitney property for states relative to $L^\infty(\bT,\sigma)$, then any Hahn--Banach extension of $\phi$ to $\Hinf$ would be absolutely continuous, and since $\AD$ is weak-$*$ dense in $\Hinf$, such an extension would be unique. We show that this is not the case.

Let $K \subset \bT$ be a closed, nowhere dense set with $\sigma(K)>0$. Let $\phi'$ be the state on $\rC(\mathbb{T})$ given by integration with respect to $\sigma$. Apply Lemma \ref{L:sarason} to find a state $\psi$ on $L^\infty(\bT,\sigma)$ extending $\phi'$ such that
\[
\psi((1-\chi_K) f)=\psi(f), \quad f\in L^\infty(\bT,\sigma).
\]
Plainly, $\rho=\psi|_{\Hinf}$ is a Hahn--Banach extension of $\phi$. Next, let $Q\in \Hinf$ be an outer function such 
  that $|Q| = \exp(\chi_K)$ almost everywhere on $\bT$. Explicitly,
  \begin{equation*}
    Q(z) = \exp \Big( \int_{0}^{2 \pi} \frac{e^{i t} + z}{e^{it} - z} \chi_K(e^{i t}) \, \frac{dt}{2 \pi} \Big);
  \end{equation*}
  see \cite[page 62]{hoffman1988}.
  Then
  \begin{equation*}
    \| (1 - \chi_K) Q\| = \| (1 - \chi_K) \exp(\chi_K)\|= 1,
  \end{equation*}
  so
  \begin{equation*}
    |\rho(Q) | = |\psi ( (1 - \chi_K) Q) | \le 1.
  \end{equation*}
  On the other hand,  the explicit formula for $Q$ yields $|\theta (Q)| = |Q(0)|  > 1$
 since $\sigma(K)>0$. Thus, $\theta\neq \rho$, as desired.
\end{proof}

\begin{remark}
  \label{rem:MOT-GW}
  The conclusion of Proposition \ref{P:GWADHinf} also follows from an example of Miller, Olin and Thomson. In \cite[Example 40]{MOT86}, these authors construct a unital homomorphism $\pi: H^\infty(\mathbb{D}) \to L^\infty(\mathbb{T},\sigma)$ such that $\pi(f) = f \big|_{\mathbb{T}}$ for all $f \in \rA(\mathbb{D})$, but such that $\pi$ is not weak-$*$ continuous. Since $L^\infty(\mathbb{T},\sigma)$ is a commutative $\rC^*$-algebra, the homomorphism $\pi$ is necessarily contractive. In particular, there exists a weak-$*$ continuous state $\psi$ on $L^\infty(\mathbb{T},\sigma)$ such that
  the state $\varphi: H^\infty(\mathbb{D}) \to \mathbb{C}$ defined by $\varphi = \psi \circ \pi$
  is not weak-$*$ continuous. But $\varphi \big|_{\rA(\mathbb{D})}$ is absolutely continuous, so the inclusion
  $\rA(\mathbb{D}) \subset H^\infty(\mathbb{D})$ does not have the Gleason--Whitney property for states relative to $L^\infty(\mathbb{T},\sigma)$.

  Alternatively, the construction on page 52 of \cite{MOT86} directly yields a state on $H^\infty(\mathbb{D})$ that is not weak-$*$ continuous, but its restriction to the disc algebra is absolutely continuous.
  \end{remark}

Proposition \ref{P:GWADHinf} can be applied to some non-uniform algebras from Subsection \ref{SS:multipliers} as well.

\begin{proposition}\label{P:AFMF}
Let $H^2_d$ be the Drury--Arveson space on  $\bB_d$. Then, the inclusion $\A(H^2_d)\subset \M(H^2_d)$ does not have the Gleason--Whitney property for states relative to $B(H^2_d)$.
\end{proposition}
\begin{proof}
The operator $M_{z_1}$ is a pure contraction on $H^2_d$, in the sense that the sequence $(M_{z_1}^{*n}h)$ converges to $0$ for every $h\in H^2_d$. It follows from \cite[Theorem II.2.1]{nagy2010} that there is a unital, contractive and weak-$*$ continuous homomorphism $\rho:\Hinf\to \M(H^2_d)$ such that
\[
(\rho(f))(z_1,\ldots,z_d)=f(z_1).
\] 
Further,  recall that $H^2_d$ consists of holomorphic functions on $\bB_d$. It is well known that there is a unital, contractive and weak-$*$ continuous homomorphism $\lambda:\M(H^2_d)\to \Hinf$  defined as
\[
\lambda(M_f)(z)=f(z,0,\ldots,0), \quad z\in \bD
\]
for every $f\in \M(H^2_d)$. It is then easily verified that $\lambda(\A(H^2_d))\subset\AD$ and that $\lambda\circ \rho=\id$.

Next,  let $\phi$ be the absolutely continuous state on $\AD$ from Proposition \ref{P:GWADHinf}, which admits a state extension $\psi$ to $\Hinf$ that is not weak-$*$ continuous. Consider the state $\widehat\phi=\phi\circ \lambda|_{\A(H^2_d)} $ on $\A(H^2_d)$, and its extension $\widehat\psi=\psi\circ\lambda$ to $\M(H^2_d)$. Because $\lambda$ is weak-$*$ continuous, the state $\widehat\phi$ extends weak-$*$ continuously to $\M(H^2_d)$, and thus is absolutely continuous relative to $B(H^2_d)$.  If the inclusion $\A(H^2_d)\subset \M(H^2_d)$ had the Gleason--Whitney property for states relative to $B(H^2_d)$, then $\widehat\psi$ would be weak-$*$ continuous on $\M(H^2_d)$, and so would be $\psi=\widehat\psi\circ \rho$,  a contradiction.
 \end{proof}

 \begin{remark}
   As mentioned previously, it is a theorem of And\^o \cite{ando1978} that $H^\infty(\mathbb{D})=\M(H^2_1)$
   admits a Lebesgue decomposition.
   We do not know if $\mathcal{M}(H^2_d)$ admits a Lebesgue  decomposition for $d \ge 2$.
   But if it does, then the Lebesgue decomposition of $\mathcal{M}(H^2_d)$
   is neither compatible with that of $\mathcal{A}(H^2_d)$, nor with that
   of $B(H^2_d)$, by Proposition \ref{P:GWcompatspaces}, Proposition \ref{P:AFMF} and Proposition \ref{P:MFTF}.
 \end{remark}

\subsection{Riemann integrable functions}\label{SS:Riemann}
In this subsection, we turn to algebras of Riemann integrable functions. Given a  bounded, Borel measurable, real-valued function on the unit circle $\bT$, we define its  \emph{lower semi-continuous envelope} $f_*:\bT\to\bR$  as the largest lower semi-continuous function less than or equal to $f$. The \emph{upper semi-continuous envelope} $f^*:\bT\to\bR$ is defined to be the smallest upper semi-continuous function greater than or equal to $f$.

\begin{lemma}\label{L:semicont}
Let $\phi$ be an absolutely continuous state on $L^\infty(\bT,\sigma)$. Let $f:\bT\to \bR$ be a bounded Borel measurable function. Then,
  \[
      \sup \{ \varphi(g) : g \in \rC(\bT), g \le f \}
    = \phi(f_*)
  \]
  and
  \[
      \inf \{ \varphi(h) : h \in \rC(\bT), h \ge f  \}
    = \phi(f^*).
  \]
\end{lemma}
\begin{proof}
By definition of the envelopes $f_*$ and $f^*$ we readily find
  \[
      \sup \{ \varphi(g) : g \in \rC(\bT), g \le f \}
    \leq  \phi(f_*)
  \]
  and
  \[
      \inf \{ \varphi(h) : h \in \rC(\bT), h \ge f  \}
    \geq \phi(f^*).
  \]
 Next, note that $f_*$ is the  pointwise supremum of an increasing
  sequence of continuous functions, and $f^*$ is the pointwise infimum of a decreasing sequence of continuous
  functions (\cite[7K.4]{willard1970}). The desired equalities thus follow from the monotone convergence theorem, as $\phi$ is given by integration against some regular Borel probability measure on $\bT$.
   \end{proof}

Let $\cR(\bT) \subset L^\infty(\bT,\sigma)$ denote the space of all functions that agree with
a Riemann integrable function almost everywhere on $\bT$. This space
is a norm closed subalgebra of $L^\infty(\bT,\sigma)$ by a result of Orlicz \cite{Orlicz29} (see also the introduction
of \cite{domanski1993}).
Our next result shows that $\cR(\bT)$ is the unique maximal subspace $\X\subset L^\infty(\bT,\sigma)$ such that the inclusion $\rC(\bT)\subset \X$ has the Gleason--Whitney property for states relative to $L^\infty(\bT,\sigma)$. 
This result is inspired by a remark preceding Theorem 2 in \cite{sarason1973}.

\begin{theorem}
  \label{T:riemann_gleason}
  Let $\X$ be a subspace of $L^\infty(\bT,\sigma)$ that contains $\rC(\bT)$.
  The following assertions are equivalent.
  \begin{enumerate}[{\rm (i)}]
    \item The inclusion $\rC(\bT)\subset \X$ has the Gleason--Whitney property for states relative to $L^\infty(\bT,\sigma)$.
     \item Every absolutely continuous state on $C(\bT)$ has a unique Hahn--Banach extension
      to $\X$.
    \item We have $\X \subset \cR(\bT)$.
  \end{enumerate}
\end{theorem}

\begin{proof}
  We begin by making a few remarks.
Observe first that the unit ball of $\rC(\bT)$ is weak-$*$ dense in that of $L^\infty(\bT,\sigma)$. In particular, this implies that, given a functional $\phi$ on $\rC(\bT)$,
\begin{enumerate}[{\rm (a)}]
\item  $\phi$ admits at most one weak-$*$ continuous extension to $L^\infty(\bT,\sigma)$, and
\item  if $\psi$ is a weak-$*$ continuous extension to $L^\infty(\bT,\sigma)$ of $\phi$, then $\|\psi\|=\|\phi\|$.
\end{enumerate}
Second, let $\S\subset L^\infty(\bT,\sigma)$ be a unital self-adjoint subspace containing $\rC(\bT)$. Then, as is well known (\cite[Proposition 6.2]{arveson2011}), a state $\phi$ on $\rC(\bT)$ admits a unique Hahn--Banach extension to $\S$ if and only if 
  \begin{equation}
    \label{eqn:unique-H-B}
    \sup \{ \varphi(g) : g \in \rC(\bT),  g \le f  \}
    = \inf \{ \varphi(h): h \in \rC(\bT), h \ge f \}
  \end{equation}
  holds for every self-adjoint element $f\in \S$.
  We now proceed to prove the equivalences.

  (i) $\Rightarrow$ (ii):  This follows immediately from property (a) above.
    
  (ii) $\Rightarrow$ (i): Let $\phi$ be an absolutely continuous state on $\rC(\bT)$. Then, there is a weak-$*$ continuous functional $\omega$ on $L^\infty(\bT,\sigma)$ extending $\phi$. By property (b),  we infer that $\omega$ is a state. By assumption, the absolutely continuous state $\omega|_\X$ is the unique Hahn--Banach extension of $\phi$ to $\X$, whence (i) follows.

  (ii) $\Rightarrow$ (iii):   Let $\S\subset L^\infty(\bT,\sigma)$ denote the unital self-adjoint subspace generated by $\X$. Let $f\in \S$ be a self-adjoint element and   let $\phi$ be an absolutely continuous state on $\rC(\bT)$.
By assumption, $\phi$ admits a unique Hahn--Banach extension to  $\X$, and hence to $\S$, so
 Equation \eqref{eqn:unique-H-B} coupled with Lemma \ref{L:semicont} yields $\phi(f_*)=\phi(f^*)$.
 Since $\phi$ was chosen to be an arbitrary absolutely continuous state, this forces  $f_* = f^*$ almost everywhere on $\bT$. It is a theorem of Carath\'eodory \cite[Satz 7]{Caratheodory18} (see also  \cite[Proposition 15.5.3]{semadeni1971}) that this equality is equivalent to $f \in \R(\mathbb{T})$. We have thus shown that all self-adjoint elements of $\S$ lie in $\R(\bT)$, whence $\X\subset \S\subset \R(\bT)$ as desired.
  
   (iii) $\Rightarrow$ (ii): Let $\phi$ be an absolutely continuous state on $\rC(\bT)$. 
   It suffices to show that $\phi$ admits a unique state extension to $\S$, as defined above. In view of \eqref{eqn:unique-H-B} and Lemma \ref{L:semicont} we fix a self-adjoint  element $f\in \S$, and we have to show that $\phi(f_*)=\phi(f^*)$.
This follows from the assumption, since $f^* = f_*$ almost everywhere by the above mentioned theorem of Carath\'eodory.
\end{proof}

We now return to the disc algebra. Taking $\mathcal{X} = H^\infty(\mathbb{D})$ below and using the fact that there are functions in $H^\infty(\mathbb{D})$ whose boundary values are not Riemann integrable, we obtain yet another explanation for Proposition \ref{P:GWADHinf}.

\begin{corollary}
  \label{C:RiemannAD}
  Let $\X \subset H^\infty(\bD)$ be a subspace
  that contains $\rA(\bD)$. Then, the following assertions are equivalent.
  \begin{enumerate}[{\rm (i)}]
    \item The inclusion $\rA(\bD)\subset \X$ has the Gleason--Whitney property for states relative to $L^\infty(\bT,\sigma)$.
    \item Every absolutely continuous state on $\rA(\bD)$ has a unique Hahn--Banach extension
      to $\X$.
    \item We have $\X \subset \cR(\bT)$.
  \end{enumerate}
\end{corollary}

\begin{proof}
  (i) $\Rightarrow$ (ii):  This follows at once from the fact that $\rA(\bD)$ is weak-$*$ dense in $H^\infty(\bD)$.

  (ii) $\Rightarrow$ (iii): Let $\S\subset L^\infty(\bT,\sigma)$ denote the closed unital self-adjoint subspace generated by $\X$. Because $\X$ contains $\AD$, it follows that $\S$ contains $\rC(\bT)$. Let $\phi$ be an absolutely continuous state on $\rC(\bT)$. Let $\psi_1$ and $\psi_2$ be two Hahn--Banach extensions of $\phi$ to $\S$. Then, $\psi_1|_{\X}$ and $\psi_2|_{\X}$ are both Hahn--Banach extensions of the absolutely continuous state $\phi|_{\AD}$, so by assumption we find $\psi_1|_{\X}=\psi_2|_{\X}$. Now, $\psi_1$ and $\psi_2$ are both states, so by construction of $\S$, this forces $\psi_1=\psi_2$. Invoking Theorem \ref{T:riemann_gleason}, we infer that $\X\subset \S\subset \R(\bT)$.
 
  (iii) $\Rightarrow$ (i): Let $\S\subset L^\infty(\bT,\sigma)$ be defined as above. 
Let $\phi$ be an absolutely continuous state on $\AD$ and let $\psi$ be a Hahn--Banach extension to $\X$.
Let $\psi'$ be a Hahn--Banach extension of $\psi$ to $\S$ and let $\varphi' = \psi' \big|_{C(\mathbb{T})}$.
Then $\varphi'$ is a Hahn--Banach extension of $\varphi$, so $\varphi'$ is absolutely
continuous by Proposition \ref{P:ballalgebra}. Theorem \ref{T:riemann_gleason} then implies that $\psi'$
is absolutely continuous, hence so is $\psi$.
\end{proof}

We now illustrate these ideas with a concrete example.

\begin{example}\label{E:ADRiemann}
Let $\theta\in \Hinf$ be an inner function. Assume that there is a closed countable set $\Gamma\subset \bT$ such that $\theta$ extends continuously to $\ol{\bD}\setminus \Gamma$; this can be arranged by specifying the cluster set of the zeros of $\theta$, along with the support of some singular measure on $\bT$ (see \cite[pages 68-69]{hoffman1988}). Let $\X_\theta \subset \Hinf$ denote the subspace generated by $\AD$ and $\theta$. By construction, we see that $\X\subset \R(\bT)$.  An application of Corollary \ref{C:RiemannAD} reveals that the inclusion $\AD\subset \X_\theta$ has the Gleason--Whitney property for states relative to $L^\infty(\bT,\sigma)$. 
\qed
\end{example}

\subsection{Multivariate variations on the classical inclusion}\label{SS:classmulti}

From the point of view of the Gleason--Whitney property, the classical inclusion $H^\infty(\bD)\subset L^\infty(\bT,\sigma)$ is the prototypical example. In this subsection, we aim to show that the corresponding multivariate inclusions, either on the ball or the polydisc, do not have the Gleason--Whitney property relative to the ambient $L^\infty$-space.
Our arguments hinge on the following technical tool.

\begin{lemma}
  \label{lem:GW_L_infty}
  Let $X$ be a compact Hausdorff space, let $\mu$ be a positive regular Borel measure on $X$ and let $\cA \subset L^\infty(X,\mu)$
  be a unital subspace. Assume that the inclusion $\cA\subset L^\infty(X,\mu)$ has the Gleason--Whitney property for states relative to $L^\infty(X,\mu)$.
  Suppose further that $\cA + \rC(X)$ satisfies the following approximation property: for every $f \in \cA + \rC(X)$,
  there exist sequences $(a_n)$ in $\cA \cap \rC(X)$ and $(b_n)$ in $\rC(X)$ such that
  \begin{enumerate}[{\rm (i)}]
    \item $\|a_n + b_n\| \le \|f\|$ for every $n$,
    \item the sequence $(a_n)$ converges to an element $a \in \cA$ in the weak-$*$ topology of $L^\infty(X,\mu)$,
    \item the sequence $(b_n)$ converges to an element $b \in \rC(X)$ in norm, and
    \item $a + b = f$.
  \end{enumerate}
  Let $\varphi$ be an absolutely continuous state on $\A$ and let $\theta$ be a state on $\rC(X)$ agreeing with $\phi$ on $\A\cap \rC(X)$.
  Then, $\theta$ is absolutely continuous.
\end{lemma}

\begin{proof}
First note that given $a,a'\in \A$ and $b,b'\in \rC(X)$ such that $a+b=a'+b'$ in $L^\infty(X,\mu)$, then $a-a'=b'-b\in \A\cap \rC(X)$ so that
$
\phi(a-a')=\theta(b'-b)
$
or
$
\phi(a)+\theta(b)=\phi(a')+\theta(b').
$
This implies that we may define a unital linear functional $\psi:\A+\rC(X)\to \bC$ as
\[
\psi(a+b)=\varphi(a) + \theta(b) \quad a \in \cA, b \in \rC(X).
\]
 We claim that $\psi$ is a contractive. To see this, let $f \in \cA + \rC(X)$ and choose
  sequences $(a_n)$ in $\cA \cap \rC(X)$ and $(b_n)$ in $\rC(X)$ as well as elements $a \in \cA$ and $b \in \rC(X)$
  satisfying Conditions (i)--(iv). Using that $\varphi$ is absolutely continuous, we see that
    \begin{equation*}
    \psi(f) = \varphi(a) +\theta(b)
    = \lim_{n \to \infty} \left(\varphi(a_n) +\theta(b_n)\right)
    = \lim_{n \to \infty} \theta(a_n + b_n),
  \end{equation*}
  hence
  \begin{equation*}
    |\psi(f)| \le \sup_{n \in \bN} ||a_n + b_n||\leq ||f||
  \end{equation*}
and the claim follows.
  
Let $\psi'$ be a Hahn--Banach extension of $\psi$ to $L^\infty(X,\mu)$, so that $\psi'$ is a state.
Because the inclusion $\cA\subset L^\infty(X,\mu)$ has the Gleason--Whitney property for states relative to $L^\infty(X,\mu)$, this forces $\psi'$ to be absolutely continuous, and in particular so is $\theta=\psi'|_{\rC(X)}$.
  \end{proof}

Here is a simple application.

\begin{proposition}\label{HinfbD}
Let $A$ denote area measure on $\bC$. Then, the inclusion $\Hinf\subset L^\infty(\ol{\bD},A)$ does not have the Gleason--Whitney property for states relative to $L^\infty(\ol{\bD},A)$.
\end{proposition}
\begin{proof}
Let $\phi$ be the absolutely continuous state on $\Hinf$ of evaluation at the origin. 
Let $\theta$ be the state on $\rC(\ol{\bD})$ defined as
\[
\theta(f)=\int_{\bT}fd\sigma, \quad f\in \rC(\ol{\bD}).
\]
Cauchy's formula shows that $\theta$ agrees with $\phi$ on $\Hinf\cap \rC(\ol{\bD})$. Furthermore, because $\sigma$ is not absolutely continuous (as a measure) with respect to $A$, the state $\theta$ is not absolutely continuous relative to $L^\infty(\ol{\bD},A)$. Hence, the desired statement follows from Lemma \ref{lem:GW_L_infty} once we show that  $\Hinf+ \rC(\ol{\bD})$ has the required approximation property. This is easily accomplished: given $f=a+b$ with $a\in \Hinf$ and $b\in \rC(\ol{\bD})$, define 
\[
a_n(z)=a((1-1/n)z), \quad z\in \bD
\]
and
\[
b_n(z)=b((1-1/n)z), \quad z\in \ol{\bD}.
\]
A standard argument reveals that the sequences $(a_n)$ and $(b_n)$ enjoy all required properties.
\end{proof}
We now give a slightly more complicated example, which shows that the classical Gleason--Whitney theorem does not extend to the ball. Recall that $\sigma_d$ denotes the rotation invariant probability measure on the unit sphere $\mathbb{S}_d$.
Just as in one variable. every function $f \in H^\infty(\mathbb{B}_d)$ has radial limits $f^*$ at $\sigma_d$-almost every point of $\mathbb{S}_d$;
see e.g. \cite[Theorem 5.6.4]{rudin2008}. Moreover, the map $f \mapsto f^*$ defines a linear isometry
from $H^\infty(\mathbb{B}_d)$ into $L^\infty(\mathbb{S}_d,\sigma_d)$, see the proof below for more explanation.
In this way, we regard $H^\infty(\mathbb{B}_d)$ as a subalgebra of $L^\infty(\mathbb{S}_d,\sigma_d)$.

\begin{proposition}\label{P:HballGW}
Let $d\geq 2$ be an integer. Then, the inclusion $H^\infty(\bB_d)\subset L^\infty( \bS_d,\sigma_d)$ does not have the Gleason--Whitney property for states relative to $L^\infty(\bS_d,\sigma_d)$. 
\end{proposition}
\begin{proof}
Let $\phi$ be the absolutely continuous state on $H^\infty(\bB_d)$ of evaluation at the origin.
Let $\nu$ be the $1$-dimensional normalized Lebesgue measure on the circle
  \begin{equation*}
    \{(z,0,\ldots,0): z \in \bT\} \subset \bS_d
  \end{equation*}
  and define the state $\theta$ on $\rC(\bT)$ as
  \[
  \theta(f)=\int_{\bS_d}fd\nu,\quad f\in \rC(\bT).
  \]
Cauchy's formula implies that $\theta$ agrees with $\phi$ on $H^\infty(\bB_d)\cap \rC(\bS_d)$.
  Since $\nu$ is not absolutely continuous with respect to $\sigma_d$,  the state $\theta$ is not absolutely continuous and the claim follows from Lemma \ref{lem:GW_L_infty}, once we verify that $H^\infty(\bB_d) + C(\bS_d)$
  has the required approximation property.

  To verify this property, we make use of the invariant Poisson integral \cite[page 41]{rudin2008}, which we denote by $P[u]$ for a function $u\in L^1(\bS_d,\sigma_d)$.
  If $a \in H^\infty(\mathbb{B}_d)$, then $P[a^*] = a$; this follows e.g.\ from \cite[Theorem 5.6.8]{rudin2008},
  as the map $a \mapsto a^*$ is injective by part (a) of that result, and $P[a^*]^* = a^*$ by part (b) of the same result. Contractivity of the invariant Poisson integral operator \cite[Theorem 3.3.4]{rudin2008} therefore
  shows that $\|a\|_{H^\infty(\mathbb{B}_d)} = \|a^*\|_{L^\infty(\mathbb{S}_d,\sigma_d)}$. From now on, we will
  identify $a$ and $a^*$.

  Now, let $f = a + b \in H^\infty(\bB_d) + \rC(\bS_d)$.
  For each positive integer $n$, define functions $a_n$ and $b_n$ by
  \begin{equation*}
    a_n(z) = P[a](( 1- 1/n) z), \quad z \in \overline{\mathbb{B}_d}
  \end{equation*}
  and
  \begin{equation*}
    b_n(z) = P[b]((1 - 1/n) \zeta), \quad \zeta \in \partial \mathbb{S}_d.
  \end{equation*}
  Since $a \in H^\infty(\mathbb{B}_d)$, we have $a_n(z) = a( (1 - 1/n) z)$ for all $z \in \overline{\mathbb{B}_d}$
  by the discussion in the preceding paragraph.
  In particular, $a_n \in H^\infty(\mathbb{B}_d) \cap \rC(\mathbb{S}_d)$, and $(a_n)$ converges
  to $a$ in the weak-$*$ topology of $L^\infty(\mathbb{S}_d, \sigma_d)$.
  Moreover, it is a standard property of the invariant Poisson integral that $(b_n)$ is a sequence in $\rC( \mathbb{S}_d)$ converging to $b$ uniformly
  on $\mathbb{S}_d$; see \cite[Theorem 3.3.4 (a)]{rudin2008}.
  Finally,
  \[(a_n+b_n)(\zeta)=P[a+b]((1-1/n)\zeta), \quad \zeta \in \mathbb{S}_d,\]
  whence contractivity of the invariant Poisson integral
  operator \cite[Theorem 3.3.4 (b)]{rudin2008} shows that
  \[
  \|a_n+b_n\|\leq \|a+b\|=\|f\|
  \]
  for each $n$.
 \end{proof}

In a similar way, we can tackle the case of the polydisc.

\begin{proposition}\label{P:HpolydiscGW}
Let $d\geq 2$ be an integer and let $\mu_d=\sigma\times \sigma\times \ldots \sigma$ denote the standard product measure on $\bT^d$. Then, the inclusion $H^\infty(\bD^d)\subset L^\infty( \bT^d,\mu_d)$ does not have the Gleason--Whitney property for states relative to $L^\infty(\bT_d,\mu_d)$. 
\end{proposition}
\begin{proof}
Let $\phi$ be the absolutely continuous state on $H^\infty(\bD^d)$ of evaluation at the origin. 
Let $\Delta:\bT\to \bT^d$ be the injective continuous function defined as
\[
\Delta(\zeta)=(\zeta,\zeta,\ldots,\zeta), \quad \zeta\in \bT.
\]
Let $\nu$ be the push-forward of the measure $\sigma$ by $\Delta$, and let $\theta$ be the state on $\rC(\bT^d)$ defined as
\[
\theta(f)=\int_{\bT^d}fd\nu, \quad f\in \rC(\bT^d).
\]
By the change of variable formula, for $f\in H^\infty(\bD^d) \cap \rC(\bT^d)$ we find
\[
\theta(f)=\int_{\bT^d}fd\nu=\int_{\bT}(f\circ \Delta) d\sigma=(f\circ \Delta)(0)=\phi(f).
\]
Since $\nu$ is not absolutely continuous with respect to $\mu_d$, the state $\theta$ is not absolutely continuous. Hence, the desired statement follows from Lemma \ref{lem:GW_L_infty}, once we show that $H^\infty(\bD^d)+ \rC(\bT^d)$ has the required approximation property. This can be done, as above, using the Poisson integral on $\bT^d$, based on \cite[Theorems 2.1.2, 2.1.3 and 2.3.1]{rudin1969}. We leave the standard verification to the reader.
\end{proof}

We remark that it does not seem to be known whether $H^\infty(\bB_d)$ or $H^\infty(\bD^d)$ admit a Lebesgue decomposition relative to the ambient $L^\infty$-space on the boundary. By virtue of Propositions \ref{P:HballGW} and \ref{P:HpolydiscGW} along with Proposition \ref{P:GWcompatspaces}, we see that if such decompositions do exist, then they cannot be compatible with that of the ambient space.

\section{Absolute continuity and homomorphisms}\label{S:homom}

Section \ref{S:examples} illustrates that the Gleason--Whitney property is somewhat rare, especially among commonly studied concrete operator algebras. In these instances, states rarely display the rigidity behaviour required by the Gleason--Whitney property. 
It is natural to wonder if maps enjoying richer structure, such as homomorphisms, exhibit this rigidity property of extensions more widely. We end the paper by tackling this question.

Let $\W$ be a von Neumann algebra and let $\A\subset \W$ be a unital operator algebra. A unital bounded homomorphism $\pi:\A\to B(\H)$ will be said to be \emph{absolutely continuous} if, for  every $x\in \H$, the functional
\[
a\mapsto \langle \pi(a)x,x\rangle, \quad a\in \A
\]
is absolutely continuous relative to $\W$. By means of the polarization identity, this is equivalent to the requirement that for every $x,y\in \H$ the functional
\[
a\mapsto \langle \pi(a)x,y\rangle, \quad a\in \A
\]
be absolutely continuous relative to $\W$.

We start with concrete example showing how to relate absolute continuity of states to that of certain homomorphisms.

\begin{example}\label{E:GNS}
As shown in the proof of Proposition \ref{P:GWADHinf}, there exists a state $\phi_0$ on $\Hinf$ that is not absolutely continuous and that satisfies
\[
\phi(f)=f(0), \quad f\in \AD.
\]
From $\varphi$, we will construct a homomorphism $\rho: \Hinf \to B(\mathcal{H})$ that is not absolutely
continuous, yet the restriction of $\rho$ to $\rA(\mathbb{D})$ is absolutely continuous.

Let $X$ denote the maximal ideal space of $L^\infty(\mathbb{T},\sigma)$.
Given a function $f \in L^\infty(\mathbb{T},\sigma)$, we denote its Gelfand transform by $\widehat{f} \in \rC(X)$.
By the Hahn--Banach theorem and the Riesz representation theorem, there exists a regular Borel probability
measure $\mu$ on $X$ such that $\varphi(f) = \int_X \widehat{f} \, d \mu$ for all $f \in \Hinf$.
Consider the representation
\begin{equation*}
  \tau: L^\infty(\mathbb{T},\sigma) \to B(L^2(\mu)), \quad f \mapsto M_{\widehat{f}},
\end{equation*}
by multiplication operators. (This is the GNS representation of the Hahn--Banach extension of $\varphi$ to $L^\infty(\mathbb{T},\sigma)$.) Finally, let $\rho$ be the restriction of $\tau$ to $\Hinf$.
Since
\begin{equation*}
  \varphi(f) = \langle \rho(f) 1, 1 \rangle, \quad f \in \Hinf,
\end{equation*}
it is clear that $\rho$ is not absolutely continuous.

It remains to see that the restriction of $\rho$ to $\rA(\mathbb{D})$ is absolutely continuous.
In fact, we will show that the restriction of $\tau$ to $\rC(\mathbb{T})$ is absolutely continuous.
To this end, let $h \in L^2(\mu)$ be a unit vector and consider the state
\begin{equation*}
  \psi: \rC(\mathbb{T}) \to \mathbb{C}, \quad f \mapsto \langle \tau(f) h, h \rangle
  = \int_{X} \widehat{f} |h|^2 \, d \mu.
\end{equation*}
There exists a Borel probability measure $\nu$ on $\mathbb{T}$ such
that $\psi(f) = \int_{\mathbb{T}} f \, d \nu$ for all $f \in \rC(\mathbb{T})$.
We have to show that $\nu$ is absolutely continuous with respect to Lebesgue measure $\sigma$ on $\mathbb{T}$.
By regularity of $\nu$, it suffices to show that whenever $K \subset \mathbb{T}$ is a compact set of Lebesgue
measure zero, then $\nu(K) = 0$.

To prove this claim, let $K$ be such a set, and let $f \in \rC(\mathbb{T})$ be such that
$f \big|_K = 1$ and $0 \le f(z) < 1$ for all $z \in \mathbb{T} \setminus K$.
Since
\begin{equation*}
  \int_{X} \widehat{f} \, d \mu = \varphi(f) = f(0) = \int_{\mathbb{T}} f \, d \sigma
\end{equation*}
holds for all $f \in \rA(\mathbb{D})$, the outer equality holds for all $f \in \rC(\mathbb{T})$, and so
\begin{equation*}
\lim_{n \to \infty} \int_{X} \widehat{f}^n \, d \mu = \lim_{n \to \infty} \int_{\mathbb{T}} f^n\, d \sigma
  = \sigma(K) = 0
\end{equation*}
by the dominated convergence theorem. Since $0 \le \widehat{f} \le 1$, it follows that $(\widehat{f}^n)_n$ converges to zero $\mu$-almost everywhere.
Thus,
\begin{equation*}
  \nu(K) = \lim_{n \to \infty} \int_{\mathbb{T}} f^n \, d \nu
  = \lim_{n \to \infty} \psi(f^n)
  = \lim_{n \to \infty} \int_{X} \widehat{f}^n |h|^2 \, d \mu = 0,
\end{equation*}
as desired.
  \qed
\end{example}

The homomorphism $\rho:\Hinf\to B(L^2(\mu))$ from the previous example has peculiar properties. Indeed, its restriction to $\AD$ is absolutely continuous, yet $\rho$ itself is not absolutely continuous on $\Hinf$. Thus, we view $\rho$ to be an instance of a pathological functional calculus for the unitary operator $\rho(z) \in B(L^2(\mu))$. Note however that the space $L^2(\mu)$ constructed above tends to be quite large, as there is no reason to expect it to be separable.

Such pathological functional calculi were examined in great depth in the work of Miller--Olin--Thompson \cite{MOT86}. Therein, by means of more sophisticated arguments, examples of these homomorphisms were constructed on separable Hilbert spaces. In fact, as mentioned in Remark \ref{rem:MOT-GW}, it can even be achieved that $\rho|_{\AD}$ coincides with the standard representation on $L^2(\bT,\sigma)$ where $\rho(z)$ is the familiar bilateral shift \cite[Example 40]{MOT86}.

Finally, we show how the concrete construction from Example \ref{E:GNS} can be adapted to more general contexts. 

\begin{theorem}\label{T:staterep}
  Let $\W$ be a von Neumann algebra and let $\A\subset \W$ be a unital subalgebra that admits a Lebesgue decomposition relative to $\W$.  Let $\pi:\A\to B(\H)$ be a unital contractive homomorphism, and let $\xi\in \H$ be a unit vector that is cyclic for the $\rC^*$-algebra generated by $\pi(\mathcal{A})$ and the commutant of $\pi(\mathcal{A})$. Let $\psi:\A\to \bC$ be the state defined as
\[
\psi(a)=\langle \pi(a)\xi,\xi\rangle,\quad a\in \A.
\]
Then, $\pi$ is absolutely continuous if and only if $\psi$ is absolutely continuous.
\end{theorem}
\begin{proof}
It is immediate from the definition that $\psi$ is absolutely continuous if $\pi$ is. 

Assume henceforth that $\psi$ is absolutely continuous. Let $\widehat\pi:\A^{**}\to B(\H)$ be the unique weak-$*$ continuous extension of $\pi$, which is also a unital contractive homomorphism; see \cite[2.5.5]{BLM2004}. By Theorem \ref{T:Lebproj}, $\A$ admits a Lebesgue projection $\fz\in \A^{**}$.
Then, $\widehat\pi(\fz)$ is a contractive idempotent, and hence it is a self-adjoint projection in $B(\H)$. Since $\fz\in \A^{**}$ is central, $\widehat\pi(\fz)$ commutes with $\pi(\mathcal{A})$.
Moreover, $\widehat{\pi}(\fz)$ commutes with the commutant $\pi(\mathcal{A})'$ of $\pi(\mathcal{A})$,
hence it commutes with $C^*(\pi(\mathcal{A}) \cup \pi(\mathcal{A})')$.

We first claim that $\widehat\pi(\fz)=I$.  Since $\xi$ is cyclic for $\rC^*(\pi(\A) \cup \pi(\mathcal{A})')$, it suffices to verify that $\widehat\pi(I-\fz)\xi=0$.
To see this, invoke Goldstine's theorem to find a net of contractions $(a_i)_i$ in $\A$ such that $(a_i)_i$ converges to $I-\fz$ in the weak-$*$ topology of $\A^{**}$. Note then that $(\fz a_i)_i$ converges to $0$ in the weak-$*$ topology of $\A^{**}$. By assumption, $\psi$ is absolutely continuous, whence $ \psi\cdot \fz=\psi$. We may thus compute
\begin{align*}
\|\widehat\pi(I-\fz)\xi\|^2&=
\langle \widehat\pi(I-\fz)\xi,\xi\rangle=\lim_i \langle \pi(a_i)\xi,\xi\rangle\\
&=\lim_i\psi(a_i)=\lim_i \psi(\fz a_i)\\
&=0
\end{align*}
as desired. Therefore, $\widehat\pi(\fz)=I$.

Finally, given $h\in \H$ we consider the functional $\phi:\A\to\bC$ defined as
\[
\phi(a)=\langle \pi(a)h,h \rangle, \quad a\in \A.
\]
For every $a\in \A$, we find that 
\begin{align*}
(\phi\cdot \fz)(a)&=\widehat{\phi}(\fz a)=\langle \widehat{\pi}(\fz a)h,h\rangle\\
&=\langle \pi( a)h,h\rangle=\phi(a).
\end{align*}
This implies that $\phi$ is absolutely continuous  by Theorem \ref{T:Lebproj}. Since $h\in \H$ was arbitrary, we infer that $\pi$ is absolutely continuous, and the proof is complete.
\end{proof}

\begin{remark}
  Let $\mathcal{A} = \rA(\mathbb{D})$, $\mathcal{W} = L^\infty(\mathbb{T},\sigma)$ and $\pi$ be the restriction
  of the homomorphism $\rho$ of Example \ref{E:GNS} to $\rA(\mathbb{D})$. Now, $\rC(\bT)$ admits a Lebesgue decomposition relative to $\W$ by Lemma \ref{L:C*decomp}. Further, the classical F.\ and M.\ Riesz theorem along with Theorem \ref{T:Lebcompat} implies that $\mathcal{A}$ admits a Lebesgue decomposition compatible with that of $\rC(\bT)$.   Next, we observe that the vector $\xi = 1 \in L^2(\mu)$ is cyclic for $\tau(L^\infty(\bT,\sigma))$, which is contained
  in the commutant of $\pi(\mathcal{A})$. Thus, Theorem \ref{T:staterep} gives another argument
  for why $\pi$ is absolutely continuous.
\end{remark}

\bibliographystyle{plain}
\bibliography{GW}

\end{document}